\documentclass{amsart}

\usepackage{xcolor}
\definecolor{nicered}{rgb}{0.6, 0, 0.1}
\definecolor{niceblue}{rgb}{0.06, 0.3, 0.57}
\definecolor{nicegreen}{rgb}{0.0, 0.51, 0.5}
\usepackage{color}

\usepackage[colorlinks=true,pagebackref,hyperindex,citecolor=nicegreen,linkcolor=niceblue,urlcolor=nicered]{hyperref}
\usepackage{amsmath,amsfonts,amssymb}
\usepackage{mathrsfs}
\usepackage{mathtools}
\usepackage{microtype}
\usepackage[all]{xy}
\usepackage[T1]{fontenc}
\usepackage{lmodern}
\usepackage{xparse}
\usepackage{enumitem}
\setlist[enumerate]{leftmargin=2em,label=\textup{\arabic*.}}
\usepackage{accents}
\newcommand\ubar[1]{%
  \underaccent{\bar}{#1}}

\usepackage{cleveref}
\crefname{equation}{Eq.}{Eqs.}
\crefname{theorem}{Theorem}{Theorems} 
\crefname{lemma}{Lemma}{Lemmas}
\crefname{corollary}{Corollary}{Corollaries}
\crefname{proposition}{Proposition}{Propositions}
\crefname{definition}{Definition}{Definitions}
\crefname{remark}{Remark}{Remarks}
\crefname{example}{Example}{Examples}
\crefname{notation}{Notation}{Notations}
\crefname{setup}{Setup}{Setup}
\crefname{question}{Question}{Question}
\crefname{convention}{Convention}{Conventions}

\newtheorem{theorem}{Theorem}[section]
\newtheorem{corollary}[theorem]{Corollary}
\newtheorem{lemma}[theorem]{Lemma}
\newtheorem{proposition}[theorem]{Proposition}

\newtheorem{question}[theorem]{Question}
\theoremstyle{definition}
\newtheorem{definition}[theorem]{Definition}

\newtheorem{example}[theorem]{Example}
\newtheorem{remark}[theorem]{Remark}
\newtheorem{notation}[theorem]{Notation}
\newtheorem{convention}[theorem]{Convention}

\newtheorem{setup}[theorem]{Setup}
\numberwithin{equation}{section}

\newtheorem{theoremx}{Theorem}

\newtheorem{corollaryx}[theoremx]{Corollary}

\newcommand{\lct}{\operatorname{lct}}	

\newcommand{\m}{\mathfrak{m}}

\newcommand{\ideal}[1]{\langle #1 \rangle}
\newcommand{\tminusf}{\seq{t}-\seq{f}}
\newcommand{\idtminusf}{\ideal{\tminusf}}
\newcommand{\wrttminusf}{_{\idtminusf}}

\newcommand{\KK}{\mathbb{K}}
\newcommand{\CC}{\mathbb{C}}
\newcommand{\QQ}{\mathbb{Q}}
\newcommand{\RR}{\mathbb{R}}
\newcommand{\cJ}{\mathcal{J}}

\newcommand{\cO}{\mathcal{O}}
\newcommand{\cF}{\mathcal{F}}

\newcommand{\Cech}{ \check{\rm{C}}}
\newcommand{\cR}{\mathcal{R}}

\newcommand{\ZZ}{\mathbb{Z}}
\newcommand{\NN}{\mathbb{N}}

\newcommand{\act}{\mathbin{\vcenter{\hbox{\scalebox{0.7}{$\bullet$}}}}}

\renewcommand{\geq}{\geqslant}
\renewcommand{\leq}{\leqslant}
\renewcommand{\ge}{\geqslant}
\renewcommand{\le}{\leqslant}

\newcommand{\Spec}{\operatorname{Spec}}
\newcommand{\Hom}{\operatorname{Hom}}

\newcommand{\ord}{\operatorname{ord}}
\newcommand{\nsupp}{\operatorname{nsupp}}

\newcommand{\Ker}{\operatorname{Ker}}

\newcommand{\IM}{\operatorname{Im}}

\newcommand{\gr}{\operatorname{gr}}

\newcommand{\CoKer}{\operatorname{Coker}}
\newcommand{\cC}{\mathcal{C}}

\newcommand{\club}{\mathbin{\vcenter{\hbox{\scalebox{0.5}{$\clubsuit$}}}}}
\newcommand{\spade}{\mathbin{\vcenter{\hbox{\scalebox{0.5}{$\spadesuit$}}}}}

\newcommand{\wrtt}{_{\ideal{\seq{t}}}}
\newcommand{\wrtf}{_{\ideal{\seq{f}}}}
\newcommand{\VV}{\mathfrak{V}}

\NewDocumentCommand \pd { m o }
{
\IfNoValueTF {#2}
{ \partial_{#1} }
{ \partial_{#1} \act #2 }
}
\NewDocumentCommand \fpd { m o }
{
\IfNoValueTF {#2}
{ \frac{\partial\ }{\partial #1 } }
{ \frac{\partial #2 }{\partial #1 } }
}

\newcommand{\fs}{\boldsymbol{f^s}}
\newcommand{\fsl}{\boldsymbol{\ubar{f}}^{\boldsymbol{\ubar{s}}}}

\newcommand{\fsll}{\boldsymbol{f_1^{s_1}} \cdots \boldsymbol{f_\ell^{s_\ell}} }

\newcommand{\R}{R}
\newcommand{\A}{A}
\newcommand{\T}{T}

\newcommand{\ti}[1]{\tau_{#1}^{\phantom |}}

\NewDocumentCommand \seq { o m }
{
\IfNoValueTF {#1}
{ \ubar{#2} }
{ {#2}_1,\ldots,{#2}_{#1} }
}

\NewDocumentCommand \pt { o m }
{
\IfNoValueTF {#1}
{ #2 }
{ ({#2}_1,\ldots,{#2}_{#1}) }
}

\usepackage[textwidth=3.3 cm,textsize=small,shadow
]{todonotes}

\newcommand{\details}[2][]{} 
 
\title[Bernstein--Sato functional equations]{Bernstein--Sato functional equations, $V$-filtrations, and multiplier ideals of direct summands}

\author[\`Alvarez Montaner et al.]{Josep \`Alvarez Montaner{$^1$}}
\address{Departament de Matem\`atiques  and  Institut de Matem\`atiques de la UPC-BarcelonaTech (IMTech)\\  Universitat Polit\`ecnica de Catalunya, Barcelona 08028, Spain} 
\email{josep.alvarez@upc.edu}

\author[]{Daniel J. Hern\'andez{$^2$}}
\address{Department of Mathematics, University of Kansas, Lawrence, KS 66045, USA}
\email{hernandez@ku.edu}

\author[]{Jack Jeffries{$^3$}}
\address{University of Nebraska-Lincoln, Lincoln, NE~68502, USA}
\email{jack.jeffries@unl.edu}

\author[]{Luis N\'u\~nez-Betancourt${^4}$}
\address{Centro de Investigaci\'on en Matem\'aticas, Guanajuato, Gto., M\'exico}
\email{luisnub@cimat.mx}

\author[]{Pedro Teixeira}
\address{Department of Mathematics, Knox College, Galesburg, IL 61401, USA}
\email{pteixeir@knox.edu}

\author[]{Emily E. Witt${^6}$}
\address{Department of Mathematics, University of Kansas, Lawrence, KS 66045, USA}
\email{witt@ku.edu}

\thanks{{$^1$}Partially supported by Generalitat de Catalunya 2017SGR-932 project and Spanish Ministerio de Econom\'ia y Competitividad MTM2015-69135-P}
\thanks{{$^2$}Partially supported by the NSF Grants DMS-1304250 and DMS-1600702.}
\thanks{{$^3$}Partially supported by the NSF Grant DMS-1606353.}
\thanks{{$^4$}Partially supported by the NSF Grant DMS-1502282, CONACYT Grant 284598, and Cátedras Marcos Moshinsky.}

\thanks{{$^6$}Partially supported by the NSF Grant DMS-1623035 and Simons Foundation Grant 582574.}


\begin{document}

\begin{abstract}
      This paper investigates the existence and properties of a Bernstein--Sato functional equation in nonregular settings. 
      In particular, we construct $D$-modules in which such formal equations can be studied. 
      The existence of the Bernstein--Sato polynomial for a direct summand of a polynomial over a field is proved in this context.  It is observed that this polynomial can have zero as a root, or even positive roots.
      Moreover, a theory of $V$-filtrations is introduced for nonregular rings, and the existence of these objects is established for what we call \emph{differentially extensible summands}.
      This family of rings includes toric, determinantal, and other invariant rings.
      This new theory is applied to the study of multiplier ideals and Hodge ideals of singular varieties.
      Finally, we extend known relations among the objects of interest in the smooth case to the setting of singular direct summands of polynomial rings. 
\end{abstract}

\maketitle
\tableofcontents
\section{Introduction}

The  theory of \emph{$D$-modules}---modules over rings of differential operators---on a smooth analytic or complex algebraic variety has been an active research area over the last forty years.  It serves as a powerful tool in solving problems arising in a wide range of mathematical disciplines, from analysis to algebraic geometry, the topology of varieties, representation theory, and commutative algebra.
This paper centers around two major fundamental constructions in $D$-module theory: 
 \emph{Bernstein--Sato polynomials} and \emph{$V$-filtrations}. 
	
The Bernstein--Sato polynomial of a holomorphic or regular function $f$ over $\CC$ 
is the monic polynomial $b(s)$ in $\CC[s]$ of least degree for which there exists a polynomial differential operator $\delta(s)$ in the indeterminate $s$ that satisfies the functional equation 
\[
\delta(s) \act f^{s+1} = b(s) f^s. 
\]
This object originated in independent constructions by Bernstein \cite{Bernstein}, to establish meromorphic extensions of distributions, and by Sato, as the $b$-function in the theory of prehomogeneous vector spaces \cite{MR595585,MR1086566}.

The existence of a nonzero polynomial satisfying the functional equation above was proved by Bernstein over polynomial rings, and extended by Bj\"ork \cite{Bjork74,Bjork79} to power series rings.
Kashiwara \cite{KashiwaraRationality} proved that the roots of the Bernstein--Sato polynomial for holomorphic functions are negative rational numbers, extending a result of Malgrange \cite{Mal74, BernsteinRationalMalgrange} for functions with isolated singularities.
Indeed, Malgrange exhibited  a relation between these roots and the eigenvalues of the monodromy of the Milnor fiber, and the rationality of the roots means that the monodromy is {quasi-unipotent}. 

Since its inception, the Bernstein--Sato polynomial has found broad applications in the study of singularities.
For instance, the roots of the Bernstein--Sato polynomial are related to the jumping numbers of multiplier ideals \cite{Kol,ELSV2004,BudurSaito05}, to spectral numbers \cite{Saito93, Budur03, Saito07, MR2377895}, and to poles of zeta functions \cite{DenefLoeser}.
The Bernstein--Sato polynomial also plays a key role in understanding algorithmic aspects of local cohomology modules and de Rham cohomology \cite{Oaku97, Uli99, Uli00, OTW00, MR1808827}.

Bernstein--Sato polynomials are closely related to the notion of a \emph{$V$-filtration} on a $D$-module along an element $f$, first introduced by Malgrange and Kashiwara \cite{MalgrangeVfil, KashiwaraVfil}.
This is a decreasing filtration on a $D$-module indexed by the rational numbers, and satisfying several conditions; see  \Cref{DefVfil,VfilM}.
In fact, a $V$-filtration can be described using a relative version of the Bernstein--Sato polynomial \cite{Sabbah87, Mebkhout_book}.
These filtrations were originally introduced to relate Deligne's nearby and vanishing cycles \cite{SGA7_2} to their corresponding $D$-modules via the Riemann--Hilbert correspondence. 

In the most general form, $V$-filtrations are known to exist for {regular holonomic} $D$-modules with {quasi-unipotent monodromy}. This condition is satisfied, for example, by the coordinate ring of a smooth variety. Quite nicely, the $D$-module-theoretic notion of $V$-filtration in this case  is essentially equivalent to the algebro-geometric notion of \emph{multiplier ideals} \cite{BudurSaito05}. 

Both the notion of the Bernstein--Sato polynomial and the $V$-filtration have been generalized to the setting of nonprincipal ideals by Budur, Musta\c{t}\u{a}, and Saito \cite{BMS2006a}.
The roots of the Bernstein--Sato polynomial in this context are still negative rational numbers, and the $V$-filtration of the coordinate ring along the ideal essentially coincides with the filtration of multiplier ideals of the ideal.
		
On varieties that are not smooth, there is a natural notion of a ring of differential operators, as defined by Grothendieck \cite{EGA}, but does not necessarily have the same favorable ring-theoretic properties. 
To start with, they are generally not generated by homotheties (maps of the form $s \mapsto rs$ for fixed $r$) and derivations, and a full description of these rings is only known in special cases. 
Moreover, rings of differential operators are in general neither left- nor right-Noetherian \cite{DiffNonNoeth}, and a theory of holonomic  $D$-modules is no longer available in this case.
For many applications, the theory of $D$-modules over singular varieties can be approached by Kashiwara's equivalence,  in which a singular variety is embedded into a smooth one \cite{KashiwaraEquiv}, so that one can consider the subcategory of $D$-modules over the smooth variety supported on the singular subvariety.
However, this approach is not always satisfactory, as the coordinate ring of the singular variety is not a $D$-module here. 

Nonetheless, multiplier ideals make sense for all normal varieties \cite{dFH}, and one might hope to develop appropriate Bernstein--Sato and $V$-filtration theories compatible with the multiplier ideal theory, in a way analogous to the smooth case, at least for a reasonable class of singularities.
	
The first major advance in this direction was due to Huneke and the first and fourth authors  \cite{AMHNB}.
They proved that every element $f$ (or more generally, ideal) in a  direct summand of a polynomial or formal power series ring  admits a  Bernstein--Sato polynomial in  a weaker sense, in which the functional equation $\delta(t) \act f^{t+1} = b(t) f^t$ is satisfied for all \emph{integer} values $t$.
Moreover, they showed that the Bernstein--Sato polynomial of $f$ \emph{considered as an element of the direct summand} divides the Bernstein--Sato polynomial of $f$ \emph{considered as an element of the polynomial ring}, but that they are not equal in general.
In fact, under the extra condition that every differential operator of the direct summand extends to the polynomial ring, equality always holds \cite[Theorem~6.11]{BJNB}.
 
In the pursuit of a theory of Bernstein--Sato polynomials in nonregular rings, previous results   \cite{AMHNB} are not fully satisfactory, in the sense that they do not realize the functional equation as a \emph{formal equality} in an appropriate $D$-module. Moreover, the lack of a theory of $V$-filtrations obstructs our understanding of the relationship between Bernstein--Sato polynomials and  multiplier ideals in nonregular contexts.

This paper develops a full theory of Bernstein--Sato polynomials and $V$-filtrations for a large class of direct summands of polynomial rings.
Namely, we concentrate on direct summands satisfying the aforementioned condition on the extensibility of differential operators, a class of rings that includes rings of invariants of finite groups, toric rings, and determinantal rings. 

\Cref{SecDirectSummand} builds the foundation for a full theory of Bernstein--Sato polynomials in nonregular rings. 
Given a field  $\KK$ of characteristic zero, consider a
$\KK$-algebra $\A$, and 
its ring of $\KK$-linear differential operators $D_{\A|\KK}$.  
For simplicity, we highlight our main results for Bernstein--Sato polynomials of a \emph{single element} $f\in \A$ here, but all results extend to those associated to a sequence of elements $\seq{f}=\seq[\ell]{f} \in \A$, and to the relative version of the Bernstein--Sato polynomial. 

The functional equation
\[\delta(s)\act f\fs=b(s) \fs\]
should be understood formally as an equality in
\[M^\A[\fs]\coloneqq \A_f [s] \fs,\]
which is the free rank-one $\A_{f}[s]$-module generated by the formal symbol $\fs$.  
That said, the specialized version of the Bernstein--Sato functional equation for direct summands \cite{AMHNB}, which reads as $\delta(t) \act f^{t+1} = b(t) f^t$ for all $t\in \ZZ$, is only a family of equalities in the localization $\A_f$.
The obstruction to the formal version in this context is the existence of a $D_{\A|\KK}[s]$-module structure on $M^\A[\fs]$, an issue that was not previously  addressed.
The major goal of \Cref{SecDirectSummand} is to determine such a structure. 

\begin{theoremx} (\Cref{wellDefinedActionOnMR: P})
   Suppose that $\A$ is
   either finitely generated over $\KK$, or  complete. 
	Then there exists a unique $D_{\A|\KK}[{s}]$-module structure on $M^\A[\fs]$ that is compatible with the $D_{\A|\KK}$-module structure on $\A_{f}$  after specialization.
\end{theoremx}

With this structure in hand, \Cref{functional_formal_special} (see also \Cref{formal_special})  establishes that the formal and the specialized versions of the Bernstein--Sato polynomial coincide.  That is,  the equality $\delta(s) \act f\fs=b(s) \fs$ holds in $M^{\A}[\fs]$ if and only if $\delta(t) \act f^{t+1} = b(t) f^t$ for all $t\in \ZZ$.
In particular, one deduces the existence of a formal Bernstein--Sato functional equation for direct summands of polynomial rings (see \Cref{Exist_BS}).

\Cref{wellDefinedActionOnMR: P} opens the door to the study of the Bernstein--Sato polynomial for a large class of $\KK$-algebras.
To this purpose,  necessary and sufficient conditions for its existence are presented in \Cref{prop_iff}, and examples are presented in which the roots of the Bernstein--Sato polynomial include zero, or even positive rational numbers.

\Cref{section: dds} introduces the class of rings for which we can develop the theory of $V$-filtrations.
When $\A$ is a direct summand of a ring $\T$, any differential operator on $\T$ yields a differential operator on $\A$ after composing with the splitting. On the other hand, we have the notion  of \emph{differentially extensible} rings \cite{BJNB} in which we are provided an ``opposite'' process, since in these rings, every  differential operator on $\A$ can be extended to a differential operator on $\T$.
This work combines both properties in the class of {\it direct summands of polynomial rings that are differentially extensible}, which contains many rings of interest, including:
\begin{itemize}
	\item  many cases of rings of invariants of finite groups (see \Cref{E-finitegroup}),
	\item toric rings (see \Cref{toricODE,toricLDE,toricODELDE}), and
	\item determinantal rings  (see \Cref{detlextl}).
\end{itemize}

When we turn our attention to modules, we consider a slight generalization of the category of differential direct summands  \cite{AMHNB} that allows us to consider $D_{\A|\KK}[{s}]$-modules. Combining this with differential extensibility builds the notion of \emph{differentially extensible summand} (see \Cref{DefDES} for details). In particular, we show in \Cref{MRfs_dds} that the $D_{\A|\KK}[{s}]$-module $M^\A[\fs]$ is a differential direct summand of the $D_{\T|\KK}[{s}]$-module $M^{\T}[\fs]$ and it is also a differentially extensible summand  (\Cref{MRfs_des}).


This fact yields a direct proof of the existence of the formal Bernstein--Sato polynomial for direct summands of a polynomial 
ring.
More importantly, through a slight generalization of a recent result of Musta\c{t}\u{a} \cite{Mustata2019}, we see that the Bernstein--Sato polynomial of a sequence of elements only depends on the ideal generated by these elements, thus extending previous results  \cite{BMS2006a} to this nonregular context.

In \Cref{Vfilt}, we develop the theory of $V$-filtrations for nonregular rings.
We follow the same ideas used in the smooth case, to axiomatically define the notion of a $V$-filtration along an ideal of a Noetherian $\KK$-algebra.
As in the smooth case, we begin by assuming that the ideal defines a smooth subvariety, and it is generated by a collection of variables (see \Cref{DefVfil,VfilM} for details).

A main result of this paper is the existence of $V$-filtrations for differentially extensible summands.

 \begin{theoremx} (\Cref{CorVfilGens})
Let $\R$ be a polynomial ring over a field $\KK$ of characteristic zero. 
Let $\A$ be a $\KK$-subalgebra of $\R$  such that $\A$ is a direct summand of $\R$, and for which the inclusion $\A \subseteq \R$ is differentially extensible.
Let $I$ be an ideal of $\A$, and $M$ a $D_{\A|\KK}$-module that is a differentially extensible summand of a regular holonomic $D_{\R|\KK}$-module $N$ that has quasi-unipotent monodromy.
Then $M$ admits a $V$-filtration along $I$.
\end{theoremx}

This result applies, in particular, when  $M$ is $\A$ itself, or any local cohomology module of $\A$. 

As a consequence,  we are able to define Hodge ideals in this singular ambient setting. We also extend known relations of Hodge ideals to minimal exponents in \Cref{Hodge}.

\begin{corollaryx}[\Cref{CorHodge}]
Suppose that $\KK$ has  characteristic zero and $\R$ be a polynomial ring over $\KK$.  Let $\A$ be a $\KK$-subalgebra that is a direct summand of $\R$, for which the inclusion $\A \subseteq \R$ is differentially extensible. 
Then for all $f\in A$ and $\lambda\in \QQ_{\geq 0}$, 
$I_k^A(f^\lambda)$ exists. 
Furthermore,
$$
I^A_k(f^\lambda)=I^R_k(f^\lambda)\, \cap \, A.
$$
\end{corollaryx}

We also use this new theory of $V$-filtrations in \Cref{Multiplier} to study multiplier ideals in differentially extensible summands.
In particular, we extend results previously obtained only for smooth varieties  \cite{BudurSaito05,BMS2006a}.

\begin{theoremx} (\Cref{ThmMultV})
Let $\R$ be a polynomial ring over $\KK$,  and let $\A\subseteq \R$ be a differentially extensible inclusion of finitely generated $\KK$-algebras, such that $\A$ is a direct summand of $\R$.	
For every ideal $I$ of $A$, and real number $\lambda\geq 0$, the following ideals coincide\textup:
\begin{enumerate}
\item $\cJ_\R((I\R)^\lambda)\cap \A$, 
\item $\bigcup_{\alpha>\lambda}V^\alpha \A$,  where the $V$-filtration is taken along $I$, and 
\item  $\{ g\in \A : \gamma>\lambda \hbox{  if } b^\A_{I,g}(-\gamma)=0\}$.
\end{enumerate}
\end{theoremx}

We study whether the equality $\cJ_\A(I^\lambda)=\cJ_\R((I\R)^\lambda)\cap \A$ holds, and appeal to  
positive characteristic methods to address this question.  It is known that  via reduction to positive characteristic, multiplier ideals reduce to test ideals  when $\A$ has KLT singularities \cite{CEMS}. We replace the notion of differential extensibility with the parallel notion of \emph{Cartier extensibility} (see \Cref{Def_Cartier_ext}) for rings of positive characteristic, and obtain the following comparison results.

\begin{theoremx}(\Cref{PropTestIdealRetraction}, \Cref{theorem-multiplier}) 
	\begin{enumerate}
		\item Let $\A\subseteq \R$ be an  extension of rings of positive characteristic, such that $\R$ is a regular $F$-finite domain and $\A$ is a Cartier extensible direct summand of $\R$.
		Then for every ideal $I$ of $\A$, and every real number $\lambda \geq 0$, 
		\[\ti{\A}(I^\lambda)=\ti{\R}((I\R)^\lambda)\cap \A.\] 
		\item Let $\R$ be a polynomial ring over a field $\KK$ of characteristic zero, and let $\A\subseteq \R$ be a differentially  extensible  inclusion, such that $\A$ is a direct summand of $\R$ and that $\A$ is finitely generated over $\KK$ with KLT singularities.
		Suppose that the reduction modulo $p$ is also Cartier extensible for each prime $p\gg 0$. Then for every ideal $I$ of $\A$ and every real number $\lambda \geq 0$, 
		\[
		\cJ_\A (I^\lambda)=\cJ_\R((I\R)^\lambda)\cap \A.
		\]
		\item Under the hypotheses of part 2 above, one has equalities \[\cJ_\A(I^{\lambda})= \bigcup_{\alpha>\lambda}V^\alpha \A = \{ g\in \A : \gamma > \lambda \hbox{  if } b^\A_{I,g}(-\gamma)=0\},\] where the $V$-filtration is taken along $I$.
	\end{enumerate}
   
\end{theoremx} 

In particular, this partially answers an open question regarding jumping numbers \cite[Question 4.15]{AMHNB} for the rings considered here.
Moreover, in \Cref{theorem-jn-roots}, we relate jumping numbers with the roots of the Bernstein--Sato polynomials, thus extending results in smooth varieties  \cite{ELSV2004,BudurSaito05,BMS2006a} to this nonregular setting.

The hypothesis of Cartier extensibility after reduction modulo a prime may not be necessary, as \Cref{ex_det} shows.
Finally, we point out that a similar comparison result \cite[Theorem~1.1]{MR2255178} was shown for $V$-filtrations and multiplier ideals in the case of an inclusion of a smooth divisor in a smooth variety. There are also analogous results to \Cref{PropTestIdealRetraction} for test ideals in special cases  \cite{MR1873384,MR3205587}.

\begin{setup}\label{setupS2}
	Throughout the paper, we  generally use the following notation.
	\begin{itemize}
		\item $\KK$ is a field,
		\item $\A$ and $\T$ denote Noetherian commutative rings containing $\KK$, and
 		\item $\R$ is a polynomial ring over $\KK$ in indeterminates $\seq[d]{x}$.
	\end{itemize}
\end{setup}
\noindent For the most part, the statements of lemmas, propositions, and theorems are self-contained, and do not rely on the conventions above, besides the fact that $\mathbb{K}$ \emph{always denotes a field}.

\section{The Bernstein--Sato functional equation}\label{SecDirectSummand}

\subsection{Rings of differential operators}

We start this section recalling some notions from the theory of rings of differential operators, as introduced  by Grothendieck  \cite[\textsection16.8]{EGA}.

A \emph{differential operator of order zero} on a ring $\T$ is defined by the multiplication by an element $s \in \T$.
Given an integer $m \geq 1$, a \emph{differential operator of order at most $m$} is an additive map  $\delta\in \Hom_\ZZ(\T,\T)$  such that the commutator $[\delta,s]=\delta s-s\delta$ is a differential operator of order at most $m-1$ for each $s\in \T$.
The set consisting of all differential operators of order at most $m$ is denoted by $D_{\T}^m$.

Sums and compositions of differential operators are themselves differential operators, hence the differential operators form a subring $D_\T$ of $\Hom_\ZZ(\T,\T)$ that admits a filtration
\[
   D^{0}_{\T} \subseteq D^{1}_{\T} \subseteq \cdots \subseteq \bigcup_{m\in\NN}D^{m}_{\T} = D_{\T}.
\]
We also consider the subring $D_{\T |\KK} \subseteq D_{\T}$ of  \emph{$\KK$-linear differential operators}.

For a polynomial or power series ring $\R$ in the variables $\seq[d]{x}$ over a field $\KK$, one has equalities \cite[\textsection16.11.2]{EGA}
\[ D^n_{\R|\KK} = \R \left\langle \frac{\partial^{a_1}_{x_1}}{a_1!} \cdots \frac{\partial^{a_d}_{x_d}}{a_d!} : a_1 + \cdots +a_d \leq n \right\rangle,\]
where the operator $\frac{\partial^{a_i}_{x_i}}{a_i!}$ is characterized by rule 
\[\frac{\partial^{a_i}_{x_i}}{a_i!}(x_1^{b_1} \cdots x_d^{b_d}) =  \binom{b_i}{a_i}  x_1^{b_1} \cdots x_i^{b_i-a_i} \cdots x_d^{b_d}.\]
If $\KK$ has characteristic zero, this agrees with the familiar description 
\[D^n_{\R|\KK} = \R \left\langle \partial^{a_1}_{x_1} \cdots \partial^{a_d}_{x_d} : a_1 + \cdots +a_d \leq n \right\rangle,\] where $\partial^{a_i}_{x_i} = a_i! \frac{\partial^{a_i}_{x_i}}{a_i!}$ is the usual iterated partial derivative operator.

Let $\R$ be a polynomial ring over a field $\KK$, and $\A = \R/ I$ for some ideal $I$  of $\R$.
The ring of $\KK$-linear differential operators of $\A$ has been described in terms of the $\KK$-linear differentials operators in $\R$ \cite[Theorem~15.5.13]{McC_Rob} (see also \cite{Milicic,Moncada}).
Namely, we have
\begin{equation}\label{opsonquot}
   D_{\A|\KK} \cong \frac{D_{\R|\KK}(-\log I)}{I D_{\R|\KK}},
\end{equation}
where $D_{\R|\KK}(-\log I) \coloneqq  \{ \delta\in D_{\R|\KK} : \delta\act I\subseteq I\}$ and the map in the ``left'' direction corresponds to restriction. The same results also hold when $\R$ is a formal power series ring.
We point out that the order of the differential operators is preserved;  that is, if $D^n_{\R|\KK}(-\log I) \coloneqq  \{ \delta\in D^n_{\R|\KK} : \delta\act I\subseteq I\}$, then we have
\begin{equation}\label{opsonquot2}
   D^n_{\A|\KK} \cong \frac{D^n_{\R|\KK}(-\log I)}{I D^n_{\R|\KK}}.
\end{equation}

When  $\KK$ has characteristic $p>0$,  every additive map is $\ZZ/p\ZZ$-linear, and thus $D_{\T}=D_{\T|(\ZZ/p\ZZ)}$.
Moreover, let $\T^{p^e} \subseteq \T$ be the subring consisting of the $p^{e}$-th powers of all elements of $\T$, and set $D^{(e)}_{\T}\coloneqq \Hom_{\T^{p^e}}(\T,\T)$.
In this case,
   \[D_{\T|\KK} \subseteq D_{\T} \subseteq \bigcup_{e\in\NN}D^{(e)}_{\T}.\]
We have  $D_{\T|\KK} = D_{\T}$ whenever $\KK$ is a perfect field.
On the other hand  $D_{\T} = \bigcup_{e\in\NN}D^{(e)}_{\T}$  when $\T$ is $F$-finite, that is, when $\T$ is finitely generated as a $\T^{p^e}$-module for some (equivalently, all) $e>0$; see \cite[Theorem~2.7]{SmithSP} and \cite[Theorem~1.4.9]{yekutieli.explicit_construction}. 

The ring structure and module theory of rings of differential operators is well understood when $\T$ is either a polynomial ring or a formal power series ring over a field, but much less is known when $\T$ is not a regular ring. As usual, a module over a ring of differential operators will mean a left module.

\subsection{The Bernstein--Sato functional equation for polynomial and power series rings} \label{SubSec Background B-S smooth}

In this subsection, we retain \Cref{setupS2}, with the additional assumption that \emph{$\KK$ has characteristic zero}.

Consider $D_{\R|\KK}[s]$, the polynomial ring in a new variable $s$ over the ring of $\KK$-linear differential operators $D_{\R|\KK}$.  
For every $f\in \R$, there exist $\delta(s)\in D_{\R|\KK}[s]$ and a nonzero polynomial $b(s)\in \KK[s]$ such that the following \emph{functional equation} is satisfied:
\begin{equation}\label{eq: BS functional eq for single poly}
\delta(s)\act f\fs=b(s) \fs.
\end{equation}

We note that this equation can be interpreted in two different senses. 
First, we may view it as a family of equations inside the module $\R_f$, indexed by $s\in \ZZ$; that is, one has $\delta(t) \act f^{t+1} = b(t) f^t$ in $\R_f$ for all $t\in \ZZ$.
We temporarily say that the polynomial $b(s)$ satisfies the \emph{specialized} functional equation in this case.
Later on, in \Cref{formal_special}, we prove that this interpretation is equivalent to the \emph{a priori} stronger interpretation of \eqref{eq: BS functional eq for single poly} as an equality in a $D$-module
\[M^\R[\fs]\coloneqq \R_f [s] \fs,\]
which is the free rank-one $\R_{f}[s]$-module generated by the formal symbol $\fs$. We say that the polynomial $b(s)$ satisfies the functional equation \emph{formally} in this case. A key point is that $M^\R[\fs]$  has a $D_{\R|\KK}[s]$-module structure given by the action of the partial derivatives as follows: for $h\in \R_f[s]$, we have
\[
\pd{x_r}[h \boldsymbol{f^s}] = 
\left( \fpd{x_r}[h] + s h f^{-1}\fpd{x_r}[f]\right)
\boldsymbol{f^s}.
\]
Henceforth in this subsection, we consider the functional equation formally.
The collection of all polynomials $b(s)$ satisfying a functional equation as in \eqref{eq: BS functional eq for single poly}, for some $\delta(s)\in D_{\R|\KK}[s]$, is an ideal in $\KK[s]$, and the unique monic generator of this ideal is known as the \emph{Bernstein--Sato polynomial} associated to $f$; we denote this polynomial as $b^\R_f(s)$.  
Alternatively, the Bernstein--Sato polynomial $b^\R_f(s)$ is the monic polynomial of smallest degree among those polynomials $b(s)$ such that $b(s)  \fs $ lies in the $D_{\R|\KK}[s]$-submodule of $M^\R[\fs]$ generated by $f  \fs$.

The existence of $b^\R_f(s)$ in the case that $\R$ is a polynomial ring is due to Bernstein \cite{Bernstein}, and this polynomial coincides with the \emph{$b$-function} in the theory of prehomogeneous vector spaces developed by Sato \cite{MR1086566, MR595585}.
Bj\"ork established the existence  of $b^\R_f(s)$ when $\R$ is a power series ring \cite{Bjork74, Bjork79}.
A fundamental property of the Bernstein--Sato polynomial is that its roots are rational numbers, and consequently $b^\R_f(s)\in \QQ[s]$, in the case that $\R$ is either a polynomial ring or a ring of convergent power series.  
This was proved by Malgrange for elements $f$ such that $\R / f\R$ has an isolated singularity \cite{Mal74, BernsteinRationalMalgrange}, and by Kashiwara in general \cite{KashiwaraRationality}.

 Budur, Musta\c{t}\u{a}, and Saito have extended the notion of a Bernstein--Sato polynomial to be associated to an \emph{ideal} in a polynomial ring $\R$ \cite{BMS2006a}, and we work in this more general framework. We point out that recent work of Musta\c{t}\u{a} \cite{Mustata2019} implies that this construction is also valid for formal power series rings. 


\begin{definition}[The module {$M^\R[\fsl]$}]\label{moduleMSFSL}
	Fix $\seq{f}=\seq[\ell]{f}\in \R$.
	Now, given indeterminates $\seq{s}=\seq[\ell]{s}$, we define the free rank-one $\R_{f_1\cdots f_\ell}[\seq{s}]$-module
	\[M^\R[\fsl]\coloneqq \R_{f_1\cdots f_\ell}[\seq{s}] \, \fsll,\]
	where $\fsll$ is a formal symbol for the generator.
	
	The $\R_{f_1\cdots f_\ell}[\seq{s}]$-module $M^\R[\fsl]$ has a natural  $D_{\R|\KK}[\seq{s}]$-module structure (and so also a $D_{\R|\KK}$-module structure), defined as follows:
	for $h\in \R_{f_1\cdots f_\ell}$, set
	\begin{equation}\label{MSFSstructure}
	\pd{x_r}[h \fsll]=\left( \fpd{x_r}[h] + h \sum^\ell_{i=1}s_i  f_i^{-1} \fpd{x_r}[f_i] \right)\,  \fsll.
	\end{equation}
	Since $D_{\R|\KK}[\seq{s}]$ is generated by $\R$, $\seq[\ell]{s}$, and the partial derivatives $\pd{x_1},\ldots,\pd{x_\ell}$, this gives a recipe for an action by any element of $D_{\R|\KK}[\seq{s}]$.
\end{definition}

\begin{definition}[The exponent specialization map $\varphi_{\pt{t}}$ on {$M^\R[\fsl]$}]\label{exponent-specialization}
	The \emph{exponent specialization map} on $M^\R[\fsl]$ associated to $\pt{t}=\pt[\ell]{t} \in \ZZ^\ell$ 
	is the $\R_{f_1\cdots f_\ell}$-linear map
	\[ \varphi_{\pt{t}}\colon M^\R[\fsl] \to \R_{f_1\cdots f_\ell}\]
	given by $\varphi_{\pt{t}}(\fsll)=f^{t_1}_1\cdots f^{t_\ell}_\ell$, and $\varphi_{\pt{t}}(s_i)=t_i$.
\end{definition}

For each $\pt{t}\in \ZZ^\ell$, $\delta \in D_{\R|\KK} [\seq{s}]$, and $v\in M^\R[\fsl]$ we have
\begin{equation} \label{specializationCompatibility: e}
\varphi_{\pt{t}}(\delta  \act v)=\delta(\pt{t}) \act \varphi_{\pt{t}}(v),
\end{equation}
showing, in particular, that $\varphi_{\pt{t}}$ is $D_{\R|\KK}$-linear.
Indeed, the verification of this equality reduces to the case of $\delta=\pd{x_r}$, in which case it follows from~\eqref{MSFSstructure}.

\begin{definition}[Bernstein--Sato functional equation for polynomial rings]
	Fix $\seq{f}=\seq[\ell]{f}\in \R$.
	For an integer $m \geq 0$, let $\binom{s_i}{m}$ denote $s_i(s_i-1)\cdots(s_i-m+1)/m! \in \KK[\seq{s}]$.
	There exists a nonzero polynomial $b(s) \in \QQ[s]$ for which $b(s_1+\cdots+s_\ell) \fsll$ is in the $D_{\R|\KK}[\seq{s}]$-submodule of $M^\R[\fsl]$ generated by the elements
	\[
	\prod_{\crampedclap{i\in\nsupp(\pt{c})}}\quad \binom{s_i}{-c_i} f^{c_1}_1\cdots f^{c_\ell}_\ell \fsll,
	\]
	where $\pt{c}=\pt[\ell]{c}$ runs over the elements of $\ZZ^\ell$ for which $|\pt{c}|\coloneqq c_1+\cdots+c_\ell=1$, and $\nsupp(\pt{c})\coloneqq\{i : c_i< 0 \}$.
	Equivalently, there exist $\delta_{\pt{c}} \in D_{\R|\KK}[\seq{s}]$, with $\delta_{\pt{c}} = 0$ for all but finitely many $\pt{c}$, for which the following \emph{functional equation} holds:
	\begin{equation}\label{FunctionalEquationIdeals}
	\begin{split}
	\sum_{\crampedclap{\substack{\pt{c} \in \ZZ^\ell \\ |\pt{c}|=1} }} \delta_{\pt{c}}\, \act\ \prod_{\crampedclap{i\in\nsupp(\pt{c})}}\quad  \binom{s_i}{-c_i} f^{c_1}_1\cdots f^{c_\ell}_\ell \fsll  =b(s_1+\cdots+s_\ell)\fsll.
	\end{split}
	\end{equation}  
\end{definition}


\begin{definition}[Bernstein--Sato polynomial associated to an ideal in a polynomial ring]\label{BSpolyringsideal}
	Let $I$ denote the ideal generated by $\seq{f}=\seq[\ell]{f}\in \R$.
	The \emph{Bernstein--Sato polynomial} $b^\R_I(s)$ is the unique monic polynomial $b(s)\in \QQ[s]$ of smallest degree satisfying the functional equation \eqref{FunctionalEquationIdeals}, for some $\delta_{\pt{c}} \in D_{\R|\KK}[\seq{s}]$.
\end{definition}

The Bernstein--Sato polynomial of $I$ does not depend on the choice of generators of $I$ \cite[Theorem~2.5]{BMS2006a}.
For a principal ideal $I = \ideal{f}$, $b^\R_I(s)$ agrees with $b^\R_f(s)$, and in general, like $b^\R_f(s)$, all roots of $b^\R_I(s)$ are negative rational numbers \cite{BMS2006a}.

\begin{example}
	Let $\R=\CC[x,y,z]$, $f_1=xy$, and $f_2=xz$. There is a functional equation for $\seq{f}=f_1,f_2$ involving the generators corresponding to the vectors $\pt{c}=(1,0)$ and $(0,1)$.
	Namely,
	\[
	\pd{x}\pd{y}[f_1 \boldsymbol{f_1^{s_1}}\boldsymbol{f_2^{s_2}}] + \pd{x}\pd{z}[f_2 \boldsymbol{f_1^{s_1}} \boldsymbol{f_2^{s_2}}]
	= (s_1 + s_2 +1)(s_1 + s_2 +2) \boldsymbol{f_1^{s_1}} \boldsymbol{f_2^{s_2}},
	\]
	so the Bernstein--Sato polynomial $b_{\ideal{xy,xz}}^{\R}$ divides $(s+1)(s+2)$.
	Using the methods from \cite[Section~4.3]{BMS2006a}, one can show that equality holds here.
\end{example}

\begin{example}
	Let $\R=\CC[x,y,z]$, $f_1 = xz^2$, and $f_2=yz^3$. There is a functional equation for $\seq{f}=f_1,f_2$ using the vectors $c=(0,1)$, $(1,0)$, $(2,-1)$, and $(3,-2)$:
	\begin{align*}
	\begin{split}
	\delta_1\act f_2\boldsymbol{f_1^{s_1}} \boldsymbol{f_2^{s_2}}
	+\delta_2\act f_1\boldsymbol{f_1^{s_1}} \boldsymbol{f_2^{s_2}}
	&+\delta_3\act \Big({s_2\atop 1}\Big)f_1^2f_2^{-1}\boldsymbol{f_1^{s_1}} \boldsymbol{f_2^{s_2}}
	+\delta_4\act \Big({s_2\atop 2}\Big)f_1^3f_2^{-2}\boldsymbol{f_1^{s_1}} \boldsymbol{f_2^{s_2}}
	\\ &= b(s_1+s_2)\boldsymbol{f_1^{s_1}}\boldsymbol{f_2^{s_2}},
	\end{split}
	\end{align*}
	where
	\begin{align*}
	b(s) &= (s+1)^2 (s+2) (s+\tfrac12) (s+\tfrac23) (s+\tfrac43),\\
	\delta_1 &= \tfrac{1}{2592}(-66 - 66 s_1 + 31 s_2 + 79 s_1 s_2 + 96 s_2^2)\,\pd{y}\pd{z}^3,\\
	\delta_2 &=
	\tfrac{1}{2592}(1350 + 3300 s_1 + 2592 s_1^2 + 648 s_1^3 + 3315 s_2 + 5128 s_1 s_2 + \\ 
	& \qquad\qquad 1944 s_1^2 s_2 + 2684 s_2^2 + 2114 s_1 s_2^2 + 915 s_2^3)\,\pd{x}\pd{z}^2,
	\\
	\delta_3 &= \tfrac{1}{1296}(-156 - 132 s_1 + 59 s_2 + 158 s_1 s_2 + 192 s_2^2)\,y\,\pd{x}^2\pd{z},\\
	\delta_4 &= \tfrac{1}{108}(3 - s_2)\,y^2\,\pd{x}^3.\\
	\end{align*}
        Using the methods from \cite[Section~4.3]{BMS2006a}, one can show that $b(s)$ is the Bernstein--Sato polynomial of $\seq{f}=f_1,f_2$.
\end{example}

The following generalization of \Cref{BSpolyringsideal} is useful for applications to birational geometry.

\begin{definition}[Relative Bernstein--Sato polynomial in a polynomial ring]\label{bspairpoly}
	Fix an ideal $I$ generated by $\seq[\ell]{f} \in \R$, and an element $g \in \R$.
	The \emph{Bernstein--Sato polynomial of $I$ relative to $g$}, denoted $b_{I,g}^\R(s)$,  is the unique monic polynomial $b(s)$ of the smallest degree satisfying the equation obtained from \eqref{FunctionalEquationIdeals}, after replacing $\fsll$ with $g \fsll$ on both sides.
\end{definition}

As with \Cref{BSpolyringsideal}, it is a fact that this notion is well defined (see \cite[Remark~2.3]{BudurNotes} and \cite[Section~2.10]{BMS2006a}). 
Moreover, if $g=1$, then $b_{I,g}^\R(s) = b_I^\R(s)$.

\subsection{The Bernstein--Sato functional equation for nonregular rings}\label{ss: Bernstein--Sato for nonregular rings} 

We now proceed to define a module analogous to the one considered in  \Cref{moduleMSFSL} and functional equations as considered in \Cref{BSpolyringsideal} for more general rings. Again, we may consider functional equations either in the \emph{specialized} or \emph{formal} sense, as in the previous subsection. The papers \cite{AMHNB} and \cite{HsiaoMatusevich} study Bernstein--Sato polynomials in the specialized sense; that is, the monic polynomial of smallest degree among those polynomials $b(s)$ for which there exists $\delta(s)\in D_{\R|\KK}[s]$ such that $\delta(t) \act f^{t+1} = b(t) f^t$  for all $t\in \ZZ$. Here, we develop a framework for formal Bernstein--Sato polynomials. The following object plays a central role in this theory.

\begin{definition}[{${M^{\A}[\fsl]}$} as an {$\A_{f_1\cdots f_\ell}[\seq{s}]$}-module]
	Let $\A$ be a $\KK$-algebra.
	Given $\seq{f}=\seq[\ell]{f}\in \A$, we define
	\[M^\A[\fsl] \coloneqq \A_{f_1\cdots f_\ell}[\seq{s}] \fsll\]
	as an $\A_{f_1\cdots f_\ell}[\seq{s}]$-module.
	Moreover, for $\pt{t}=\pt[\ell]{t} \in \ZZ^\ell$, we define the exponent specialization map $\varphi_{\pt{t}}\colon M^\A[\fsl]\to \A_{f_1 \cdots f_\ell}$ in the manner directly analogous to that defined in  \Cref{exponent-specialization}.
\end{definition}

We want to endow $M^\A[\fsl]$ with a $D_{\A|\KK}[\seq{s}]$-module structure that is compatible with the $D_{\A|\KK}$-module structure on $\A_{f_1 \cdots f_\ell}$ via the exponent specialization maps.
In general, this allows for at most one such structure, as \Cref{wellDefinedActionOnMR: P} shows. 
The following lemma is a key observation in establishing this.
Although this is likely a well-known result, we include its simple proof, for its centrality in our constructions, and for lack of an appropriate reference.

\begin{lemma}\label{lemma-evaluation}
   Let $\T$ be an algebra over a field of characteristic zero, and let $h\in \T[\seq[\ell]{s}]$.
   Suppose $h$ has degree at most $m$ in each variable.
   Then $h$ vanishes on the set $\mathcal{A}=\{0,1, \ldots, m\}^\ell$ if and only if $h = 0$.
\end{lemma}

\begin{proof}
   We prove the ``only if'' implication, starting with polynomials in one variable.
   The result is clear for $m=0$, so suppose $m$ is positive and the result holds for one-variable polynomials of degree strictly less than $m$.
   Suppose $h\in \T[s]$ has degree at most $m$ and vanishes on $\{0,\ldots,m\}$.
   Then $h=(s-m)g$, for some $g\in \T[s]$ of degree strictly less than $m$.
   Since $h(a)=0$ and $a-m$ is a unit in $\T$ for each $a\in \{0,\ldots,m-1\}$, it follows that $g(a)=0$ for all such $a$.
   By our induction hypothesis, $g=0$, and consequently $h=0$.

   Now suppose the result holds for polynomials in $\ell-1$ variables, and suppose $h\in \T[\seq[\ell]{s}]$ is nonzero and has degree at most $m$ in each variable.
   Without loss of generality, we assume that $m$ is positive and the indeterminate $s_\ell$ effectively appears in $h$, and write
   \[
      h = h_0 + h_1 s_\ell + \cdots + h_n s_\ell^n,
   \]
   with $h_i\in \T[\seq[\ell-1]{s}]$, for $i=0,\ldots,n$, and $h_n\ne 0$.
   By our induction hypothesis, there exist $a_1,\ldots,a_{\ell-1}\in \{0,\ldots,m\}$ such that $h_n(a_1,\ldots,a_{\ell-1})\ne 0$.
   We conclude that $h(a_1,\ldots,a_{\ell-1},s_\ell)$ is a nonzero polynomial, and there exists $a_\ell\in \{0,\ldots, m\}$ such that $h(a_1,\ldots,a_{\ell-1},a_\ell)\ne 0$. 
\end{proof}

\begin{corollary}\label{corollary-evaluation}
   Under the assumptions of \Cref{lemma-evaluation}, the polynomial $h$ has coefficients in $\QQ[h(\mathcal{A})]\subseteq \QQ[h(\ZZ^\ell)]$.
\end{corollary}

\begin{proof}
   This follows from the identity
   \[
      h=\sum_{a\in\mathcal{A}} h(a)\  \prod_{i=1}^\ell\  \prod_{\substack{0\le j\le m\\ j\ne a_i}} \frac{s_i-j}{ a_i-j },
   \]
   which is obtained by applying \Cref{lemma-evaluation} to the difference between the left- and right-hand sides.
\end{proof}

Our next technical lemma deals with the behaviour of differential operators with respect to localization. We remark that,
 given $\delta \in D_{\T|\KK}$ and $g\in \T$, there is a unique $\delta_g\in D_{\T_g|\KK}$ that extends $\delta$ \cite[Theorem~2.2.10]{Masson}.

\begin{lemma}\label{lemma-logloc}
    Given a $\KK$-algebra  $\T$, fix $g \in \T$ and an ideal $I$ of $\T$. Then every element of $D_{\T | \KK}(-\log I)$, when considered as an element of $D_{\T_g | \KK}$ by localization, is in $D_{\T_g | \KK}(-\log I \T_{g})$.
\end{lemma}

\begin{proof} First observe that if $h\in \T$ and $\delta \in D^n_{\T | \KK}(-\log I)$ for some $n \geq 1$, then for all $a\in I$, $[\delta,h]\act a=\delta\act (ha)-h\delta \act a$ is again in $I$.  Thus, $[\delta,h]\in D^{n-1}_{\T | \KK}(-\log I)$. 
	   
   For elements of $D_{\T | \KK}(-\log I)$ of order zero, the statement clearly holds.  Proceeding by induction on the order of a differential operator, fix $n \geq 1$ and $\delta \in D^n_{\T | \KK}(-\log I)$, and suppose that the claim holds for all operators in $D^{n-1}_{\T | \KK}(-\log I)$. 
   Since $ \delta \circ g^t = g^t \circ \delta + [\delta,g^t]$, it follows that for $a \in I$ and $t \geq 0$,
   \[ \delta \act a = \delta\act\left(g^t\cdot \frac{a}{g^t}\right) = g^t \delta\act \frac{a}{g^t} + [\delta,g^t] \act\frac{a}{g^t}, \quad \text{so} \quad \delta\act\frac{a}{g^t} = \frac{\delta\act a - [\delta,g^t]\act\frac{a}{g^t}}{g^t}.\]
Notice that $[\delta,g^t]$ is in $D^{n-1}_{\T | \KK}(-\log I)$ by our initial observation, so that it is also in $D_{\T_g | \KK}(-\log I \T_{g})$ by the inductive hypothesis.  
Finally, since $\delta\act a \in I$ by our choice of $\delta$, we conclude that $\delta\act \frac{a}{g^t} \in I \T_g$. 
\end{proof}

\begin{theorem} \label{wellDefinedActionOnMR: P}
   Let $\A$ be
   an algebra over a field $\KK$ of characteristic zero.
   Suppose that either $\A$ is finitely generated over $\KK$, or that $\A$ is complete.
   Given $\seq{f}=\seq[\ell]{f}\in \A$,  there is a unique $D_{\A|\KK}[\seq{s}]$-module structure on $M^\A[\fsl]$ with the following property\textup:
   for all $v\in M^\A[\fsl]$, $\delta  \in D_{\A|\KK}[\seq{s}]$, and $t\in \ZZ^{\ell}$, 
	\[\varphi_{\pt{t}}(\delta \act v) = \delta(\pt{t}) \act \varphi_{\pt{t}}(v). \]
\end{theorem}

\begin{proof}
		First write $\A=\R /I$ for some polynomial or power series ring $\R $ over $\KK$, and $I$ a radical ideal  $\R$. 
		We claim that if $\delta \in D_{\R|\KK}(-\log I)[\seq{s}]$, then 
				\begin{equation} \label{delta-I: e}
		 \delta  \act I M^\R[\fsl] \subseteq I M^\R[\fsl].
		\end{equation}
Toward \eqref{delta-I: e}, fix $a \in I \R_{f_1\cdots f_\ell}[\seq{s}]$, so that $a \fsl \in I M^{\R}[\fsl]$. Observe that for $t\in \ZZ^{\ell}$, $\delta(t)\in D_{\R|\KK}(-\log I)$ and $a(t) \in I \R_{f_1\cdots f_\ell}$, so that $\delta(t) \in D_{\R_{f_1\cdots f_\ell}|\KK}(-\log I \R_{f_1\cdots f_\ell})$ by \Cref{lemma-logloc}. Therefore, 
$\varphi_t(\delta \act a \fsl) = \delta(t) \act a(t) f_1^{t_1} \cdots f_{\ell}^{t_{\ell}}$ is in $I \R_{f_1\cdots f_\ell}$. Writing $\delta \act a \fsl = b \fsl$ for some $b\in \R_{f_1\cdots f_\ell}[\seq{s}]$, consider the image of $b$ in the polynomial ring over the
$\KK$-algebra $\A_{f_1\cdots f_\ell} = (\R / I)_{f_1\cdots f_\ell}$. As a polynomial over this ring, our work thus far shows that $b$ takes the value zero for all $t\in \ZZ^{\ell}$, and hence is the zero polynomial by \Cref{lemma-evaluation}. That is, $\delta \act a \fsl \in I \R_{f_1 \cdots f_{\ell}}[\seq{s}] \fsl$, proving \eqref{delta-I: e}.
		
	 Note also that if $\delta \in I D_{\R|\KK}$, then $\delta \act M^\R[\fsl] \subseteq I M^\R[\fsl]$. Now, using  \eqref{opsonquot}, we obtain a $D_{\A|\KK}[\seq{s}]$-module structure on $M^\R[\fsl]/ I M^\R[\fsl] \cong M^\A[\fsl]$ that is compatible with the specialization maps.
		
		To see the uniqueness of this structure, let $\club$ and $\spade$ denote two $D_{\A|\KK}[\seq{s}]$-actions that satisfy the given property.
		Fix $v$ and  $\delta$ as in the statement, and then fix $a, b\in \A_{f_1\cdots f_\ell}[\seq{s}]$ for which
		\[
		\delta \club v = a\, \fsll \text{ and } \delta \spade v = b\, \fsll.
		\]
		Since, by our hypothesis,
		\[
		a(\pt{t})f_1^{t_1}\cdots f_\ell^{t_\ell}=\varphi_{\pt{t}}(\delta \club v) = \delta(\pt{t}) \act \varphi_{\pt{t}}(v)=\varphi_{\pt{t}}(\delta \spade v)=b(\pt{t})f_1^{t_1}\cdots f_\ell^{t_\ell}
		\]
		for every $\pt{t}=\pt[\ell]{t}\in \ZZ^\ell$, we have that $a(\pt{t})=b(\pt{t})$ for every $\pt{t}\in\ZZ^\ell$.
		\Cref{lemma-evaluation} then allows us to conclude that $a=b$, and the claim follows. 
\end{proof} 

\begin{convention}\label{convention: D-module structure}
   When referring to a $D_{\A|\KK}[\seq{s}]$-module structure on $M^\A[\fsl]$, we use the unique $D_{\A|\KK}[\seq{s}]$-module structure compatible with specialization described in \Cref{wellDefinedActionOnMR: P}.   
\end{convention}

From \Cref{wellDefinedActionOnMR: P}, we obtain a suitable $D$-module in which to judge the existence of a Bernstein--Sato polynomial in the formal sense. We show now that the notions of formal and specialized Bernstein--Sato polynomials are equivalent.

\begin{proposition}\label{functional_formal_special}
   Let $\A$ be
   an algebra over a field $\KK$ of characteristic zero.
   Suppose that $\A$ is either finitely generated over $\KK$, or that $\A$ is complete.
   Then for $f\in \A$, the equality $\delta \act f\fs=b(s) \fs$ holds in $M^{\A}[\fs]$ if and only if $\delta(t) \act f^{t+1} = b(t) f^t$ for all $t\in \ZZ$. 
\end{proposition}

\begin{proof}
   If $\delta \act f\fs=b(s) \fs$ holds in $M^{\A}[\fs]$, we can apply the map $\varphi_t$ to both sides to obtain $\delta(t) \act f^{t+1} = b(t) f^t$ for any $t$.
   Conversely, suppose that $\delta(t) \act f^{t+1} = b(t) f^t$ for all $t$.
   We then have $\varphi_t (\delta \act f \fs) = \varphi_t( b(s) \fs)$ for all $t$.
   Writing $(\delta f - b(s)) \act \fs = a \fs$ for some $a\in \A_{f}[s]$, we have $a(t)=\varphi_t(a)=0$ for all $t$, so $a=0$ by \Cref{lemma-evaluation}.
   It follows that $\delta \act f \fs = b(s) \fs$ in $M^{\A}[\fs]$.
\end{proof}

An analogous (omitted) proof also holds for the general case we present next.

\begin{proposition}\label{functional_formal_special_ideal}
   Let $\A$ be
   an algebra over a field $\KK$ of characteristic zero. Suppose that $\A$ is either finitely generated over $\KK$, or that $\A$ is complete. Fix $\seq{f}=\seq[\ell]{f}\in \A$.
   The equality
   \[
      \sum_{\crampedclap{\substack{\pt{c} \in \ZZ^\ell \\ |\pt{c}|=1} }} \delta_{\pt{c}}\, \act\ \prod_{\crampedclap{i\in\nsupp(\pt{c})}}\quad  \binom{s_i}{-c_i} f^{c_1}_1\cdots f^{c_\ell}_\ell \fsll  =b(s_1+\cdots+s_\ell)\fsll
   \]
   holds in $M^{\A}[\fsl]$ if and only if
   \[
      \sum_{\crampedclap{\substack{\pt{c} \in \ZZ^\ell \\ |\pt{c}|=1} }} \delta_{\pt{c}}(\pt{t})\, \act\ \prod_{\crampedclap{i\in\nsupp(\pt{c})}}\quad  \binom{t_i}{-c_i} f^{t_1+c_1}_1\cdots f^{t_\ell + c_\ell}_\ell  =b(t_1+\cdots+t_\ell) f^{t_1}_1\cdots f^{t_\ell}_\ell
   \]
   for all $\pt{t}=\pt[\ell]{t}\in \ZZ^\ell$.
   \qed
\end{proposition}

\begin{corollary} \label{formal_special}
   Let $\A$ be
   an algebra over a field $\KK$ of characteristic zero. Suppose that $\A$ is either finitely generated over $\KK$, or that $\A$ is complete. Fix $\seq{f}=\seq[\ell]{f}\in \A$. Then there exists a formal Bernstein--Sato polynomial for $\seq{f}$ if and only if there exists a specialized Bernstein--Sato polynomial for $\seq{f}$, and in this case, these two polynomials agree.
   \qed
\end{corollary}

The existence of the specialized Bernstein--Sato polynomial for direct summands of polynomial 
rings was recently proven \cite{AMHNB}. Thus we can extend their result to the formal setting.

\begin{theorem}[{\cite[Theorem~3.14]{AMHNB}}] \label{Exist_BS}
Let $\A$ be a finitely generated $\KK$-subalgebra of a polynomial 
ring $\R$, such that $\A$ is a direct summand of $\R$. Then the Bernstein--Sato polynomial $b_{\seq{f}}^{\A}(s)$ exists for any $\seq{f}=\seq[\ell]{f}\in \A$.  Moreover $b_{\seq{f}}^{\A}(s)$
divides $b_{\seq{f}}^{\R}(s)$.
\qed
\end{theorem}

The question whether the Bernstein--Sato polynomial $b_{\seq{f}}^{\A}(s)$ only depends on the ideal generated by $\seq{f}$ was not considered in the initial work on direct summands \cite{AMHNB}.
We address this issue in \Cref{BS_des}.

\subsection{Some existence results and examples of Bernstein--Sato polynomials}
A sufficient condition for the existence of  Bernstein--Sato polynomials in the polynomial ring case is the fact that $M^{\R}[\fs] \otimes_{\KK[s]} \KK(s)$ is a $D$-module of finite length.
This result was extended to the case of differentiably admissible $\KK$-algebras \cite{MNM91, NBRingsDifType}.
In general, for the case of $\KK$-algebras, we provide necessary and sufficient conditions for the existence of Bernstein--Sato polynomials.

\begin{proposition} \label{prop_iff}
   Let $\A$ be
   an algebra over a field $\KK$ of characteristic zero.
   Suppose that $\A$ is either finitely generated over $\KK$, or that $\A$ is complete.
   Fix an element $f\in \A$.
   Then, the following are equivalent\textup:
	\begin{enumerate}
		\item There exists a Bernstein--Sato polynomial for $f$.
		\item $M^{\A}[\fs] \otimes_{\KK[s]} \KK(s)$ is a finitely-generated $D_{\A|\KK}[s] \otimes_{\KK[s]} \KK(s)$-module.
		\item $M^{\A}[\fs] \otimes_{\KK[s]} \KK(s)$ is generated by $\fs$ as a $D_{\A|\KK}[s] \otimes_{\KK[s]} \KK(s)$-module.
	\end{enumerate}
\end{proposition}
     
\begin{proof}
   We observe first that for $\delta(s)\in D_{\A|\KK}[s]$ and $b(s)\in \KK[s]$ and for any integer $j$, the equality $\delta(s) \act f \fs = b(s) \fs$ holds in $M^{\A}[\fs]$ if and only if $\delta(s+j) \act f^{j+1} \fs = b(s+j) f^{j} \fs$; this is a straightforward application of \Cref{lemma-evaluation}, as in the results of the previous subsection.

   Now, assume that $M^{\A}[\fs] \otimes_{\KK[s]} \KK(s)$ is finitely generated over  $D_{\A|\KK}[s] \otimes_{\KK[s]} \KK(s)$, and take generators $a_1 f^{j_1} \fs, \dots, a_n f^{j_n} \fs$ with $a_i\in \A$ and $j_i\in \ZZ$.
   We can replace each $j_i$ by $j=\min \{j_i\}$.
   Then, there exist operators $\delta_1,\dots,\delta_n \in D_{\A|\KK}(s)$ such that $\sum_i  \delta_i \act (a_i f^j \fs) = f^{j-1} \fs$.
   We can rewrite this as $(\sum_i \delta_i a_i) \act f^j \fs = f^{j-1} \fs$.
   The element $\sum_i \delta_i a_i$ can be written as a quotient of an element $\delta(s)\in D_{\A|\KK}[s]$ and an element $b(s)\in \KK[s]$.
   We then have the equality $\delta(s) \act f^j \fs = b(s) f^{j-1} \fs$ in $M^{\A}[\fs]$.
   By our observation above, there exists a Bernstein--Sato polynomial for $f$.

   On the other hand, if there exists a Bernstein--Sato polynomial $b(s)$ for $f$, then $\delta(s) \act f \fs = b(s) \fs$, for some $\delta(s)\in D_{\A|\KK}[s]$, and by the observation above, $\delta(s+j) \act f^{j+1} \fs = b(s+j) f^{j} \fs$ for any integer $j$.
   We then have $\frac{\delta(s+j)}{b(s+j)} \act f^{j+1} \fs = f^{j} \fs$ for any integer $j$.
   Since $M^{\A}[\fs] \otimes_{\KK[s]} \KK(s)$ is generated as an $\A$-module by $\{f^{j} \fs : j\in \ZZ\}$, we conclude that $M^{\A}[\fs] \otimes_{\KK[s]} \KK(s)$ is generated by $\fs$ as a module over $D_{\A|\KK}[s] \otimes_{\KK[s]} \KK(s)$.
\end{proof}

Next we show that for nonregular rings we have examples where the roots of the Bernstein--Sato polynomial may be zero or positive rational numbers. 

\begin{example}
   Let $\A=\CC[x,y]/\ideal{xy}$ and $f=x$.
   The operator $\delta=x \partial_x^2 \in D_{\CC[x,y] | \CC}$ is easily seen to stabilize the ideal $\ideal{xy}$, and hence descends to a differential operator on $\A$.
   The equation $\delta \act x^{t+1} = (t+1) t x^t$ holds for all nonnegative integers $t\in \NN$, and thus, by \Cref{functional_formal_special}, the functional equation $\delta \act x \boldsymbol{x^{s}} = (s+1) s \boldsymbol{x^s}$ holds in $M^\A[\boldsymbol{x^s}]$.
   Thus, $b_{x}^{\A}(s)$ divides $s(s+1)$.

   In fact, the equality $b_{x}^{\A}(s)=s(s+1)$ holds here.
   To see this, it suffices to show that if $\delta_{-1} \act x^0 = c_{-1} x^{-1}$ and $\delta_{0} \act x^1 = c_{0} x^{0}$ with $\delta_{-1},\delta_{0}\in D_{\A|\CC}$ and $c_{-1},c_0\in \CC$, then $c_{-1}=c_0=0$.
   Observe that $\A$ is a $D_{\A | \CC}$-module, that $\ideal{x}$ and $\ideal{y}$ are $D_{\A | \CC}$-submodules of $\A$  \cite[Lemma~3.3]{BJNB}, and thus the sum $\ideal{x,y}\subseteq \A$ is a $D_{\A | \CC}$-submodule.
   The claim then follows.
\end{example}

\begin{example}
   Let ${\A}=\CC[x^2,x^3]$ and $f=x^2$.
   Consider the endomorphism of ${\A}$ given by
   \[ \delta = (x \partial_x -1) \circ \partial_x^2 \circ (x \partial_x -1)^{-1},\]
   where $(x \partial_x -1)^{-1}$ is the inverse function of $x \partial_x -1$ on $\A$ (i.e., the linear operator sending $x^n$ to $\frac{1}{n-1} x^n$ for all $n\neq 1$).
   One can verify by the definition that $\delta\in D^2_{\A|\CC}$; alternatively, see \cite{PaulSmith,SmithStafford}.
   The equation $\delta \act x^{2(t+1)} = (2t+2)(2t-1) x^{2t}$ holds for all nonnegative integers $t\in \NN$, and thus, by \Cref{functional_formal_special}, the functional equation 
   \[\delta \act x^2 \boldsymbol{(x^2)^{s}} = (2s+2)(2s-1) \boldsymbol{(x^2)^s}\]
   holds in $M^\A[\boldsymbol{(x^2)^s}]$.
   Thus, $b_{x^2}^{\A}(s)$ divides $(s-\frac{1}{2})(s+1)$.

   In fact, the equality $b_{x^2}^{\A}(s)=(s-\frac{1}{2})(s+1)$ holds here.
   By the same argument as in the previous example, $s=-1$ is a root of $b_{x^2}^{\A}(s)$.
   By a specific description of $D_{\A|\CC}$  \cite{PaulSmith,SmithStafford}, every differential operator of degree $-2$ on $\A$ can be written as $(x \partial_x -1) \circ \partial_x^2 \circ \Upsilon \circ (x \partial_x -1)^{-1}$, for some $\Upsilon \in \CC[x \partial_x]$.
   Since $M^\A[\boldsymbol{(x^2)^s}]$ is a graded module we can decompose the functional equation as a sum of homogeneous pieces, and thus there exists a functional equation of the form
   \[P(s) \act x^2 \boldsymbol{(x^2)^{s}} = b^{\A}_{x^2}(s) \boldsymbol{(x^2)^s},\]
   where $b^{\A}_{x^2}(s)$ is the Bernstein--Sato polynomial, and $P(s)$ is a polynomial expression in $s$ with coefficients consisting of differentials operator of degree $-2$ on $\A$.
   Using the description of such operators given above, it is clear that $s=\frac{1}{2}$ must be a root of $b^{\A}_{x^2}(s)$.
\end{example}

Finally we remark that we can find rings where the Bernstein--Sato polynomial does not exist. The ring $\A=\CC[x,y,z]/ \ideal{x^3+y^3+z^3}$ is such an example (see \cite[Example~3.19]{AMHNB} for details).

\section{Differential direct summands and differential extensibility}\label{section: dds}

\subsection{Differential extensibility}

In this section we introduce the class of rings at the focal point of this paper, the \emph{differentially extensible} subrings of polynomial rings.
This notion was formalized to study \emph{differential signature} \cite{BJNB}, but was implicitly used earlier by Levasseur and Stafford \cite{LS} and Schwarz \cite{Schwarz}.

\begin{definition}[Differential extensibility {\cite[Definition~6.1]{BJNB}}]
   Consider a $\KK$-algebra homomorphism $\varphi\colon \A\to \T$.
   \begin{enumerate}
      \item If $\delta\in D_{\A|\KK}$, we say that a differential operator $\widetilde{\delta}\in D_{\T|\KK}$ \emph{extends} $\delta$ (or call $\widetilde{\delta}$ an \emph{extension} of $\delta$) if $\varphi \circ \delta = \widetilde{\delta}\circ \varphi$; i.e., the following diagram commutes.
      \[
         \xymatrix{
            \A\ar[d]_{\delta} \ar[r]^{\varphi} & \T\ar[d]^{\widetilde{\delta}}\\
            \A \ar[r]^{\varphi}  & \T }
      \]
      If $\varphi$ is simply an inclusion map, this is equivalent to the condition that $\widetilde{\delta}|_\A=\delta$.
      \item The map $\varphi$ is \emph{differentially extensible} if for every $\delta\in D_{\A|\KK}$,  there exists $\widetilde{\delta}\in D_{\T|\KK}$ that extends $\delta$.
      \item The map $\varphi$   is \emph{differentially extensible with respect to the order filtration}, or \emph{order-differentially extensible}, if for every nonnegative integer $n$, if $\delta\in D^n_{\A|\KK}$, then there exists $\widetilde{\delta}\in D^n_{\T|\KK}$ that extends $\delta$.
      \item If $\KK$ has positive characteristic, we say that $\varphi$  is \emph{differentially extensible with respect to the level filtration}, or \emph{level-differentially extensible}, if for every nonnegative integer $e$, given $\delta\in D^{(e)}_{\A|\KK}$, there exists $\widetilde{\delta}\in D^{(e)}_{\T|\KK}$ that extends $\delta$.
    \end{enumerate}  
\end{definition}

\begin{remark}
   If $\varphi\colon \A\to \T$ is as above, and $\seq{s}=\seq[\ell]{s}$ are indeterminates, we say that an element $\widetilde{\delta}\in D_{\T|\KK}[\seq{s}]$ \emph{extends} an element $\delta\in D_{\A|\KK}[\seq{s}]$ if
   \[(\varphi\otimes_{\KK}\KK[\seq{s}]) \circ \delta = \widetilde{\delta}\circ (\varphi\otimes_{\KK} \KK[\seq{s}]).\]
   We observe that an element $\widetilde{\delta}\in D_{\T|\KK}[\seq{s}]$ extends $\delta\in D_{\A|\KK}[\seq{s}]$  if and only if each of its coefficients is an extension of the corresponding coefficient of $\delta$.
\end{remark}

Several maps of interest are differentially extensible. 

\begin{example}[{\cite{Kantor}, \cite[Proposition~6.4, Theorem~7.1]{BJNB}}]\label{E-finitegroup}
   Let $\R$ be a polynomial ring over  $\KK$, and let $G$ be a finite group acting by degree-preserving automorphisms on $\R$ such that
   \begin{enumerate}
      \item $|G|$ is a unit in $\KK$, and
      \item No element of $G$ fixes a hyperplane in $[\R]_1$.
   \end{enumerate}
   Then the inclusion $\R^G \subseteq \R$ is order-differentially extensible.
   Moreover, if $\KK$ has positive characteristic, then the inclusion is also level-differentially extensible.
\end{example}

\begin{example}[{\cite[Proof of Theorem~3.4]{HsiaoMatusevich}, \cite{Musson}, \cite[Lemma~7.3]{BJNB}}]\label{toricODE}
   Let $\A$ be a pointed normal affine semigroup ring over a field $\KK$ of characteristic zero.
   Then there is an embedding of $\A$ into a polynomial ring $\R$ as a monomial subalgebra, such that $\A$ is a direct summand of $\R$, and for which the embedding is order-differentially extensible. 
\end{example}

\begin{lemma}\label{toricLDE}
   Suppose that $\KK$ has characteristic $p>0$, and $d$ is a positive integer.
   Let $V\subseteq \RR^d$ be a rational linear subspace, and $L\subseteq \ZZ^d$ be a lattice of finite index.
   Suppose that $p$ does not divide the index $[\ZZ^d :L]$, and that the image of the projection of $L$ onto each coordinate surjects onto $\ZZ$.
   Then the inclusion of normal semigroup rings $\KK[\NN^d \cap V \cap L] \subseteq \KK[\NN^d]$ is level-differentially extensible.
\end{lemma}

\begin{proof}
   We briefly describe our plan:  First, we show that $\KK[\NN^d \cap V \cap L]\subseteq \KK[\NN^d \cap L]$ and $\KK[\NN^d \cap L]\subseteq \KK[\NN^d]$ are each level-differentially extensible inclusions.
   From here, it follows that the composition is also  level-differentially extensible.

   It can be shown that the inclusion $\KK[\NN^d \cap V \cap L]\subseteq \KK[\NN^d \cap L]$ is level-differentially extensible by the same argument used to compute differential signature \cite[Proposition~6.6]{BJNB}.
   By our assumptions on $L$,  $\KK[\NN^d \cap L]\subseteq \KK[\NN^d]$ is the inclusion of a ring of invariants of a finite diagonal abelian group.
   Any diagonal element $g$ that fixes a hyperplane in the space of one-forms must fix every variable except one, say $x_1$, and multiply that element by a $t$-th root of unity.
   The invariants of $g$ must then be contained in the subring $\KK[x_1^t, x_2,\dots,x_d]$, contradicting the assumption that the image of $L$ under the first coordinate map surjects onto $\ZZ$.
   Thus, the inclusion $\KK[\NN^d \cap L]\subseteq \KK[\NN^d]$ is level-differentially extensible by \Cref{E-finitegroup}.
\end{proof}

\begin{remark}\label{toricODELDE}
   An embedding defined by  Hsiao and Matusevich  \cite[Notation~1.2]{HsiaoMatusevich} satisfies the hypotheses of the previous lemma for all but finitely many characteristics $p$, and as a consequence, every normal affine semigroup ring of characteristic zero can be realized as a direct summand of a polynomial ring by a map that is order-differentially extensible, and such that the reduction modulo $p$ of this map for all but finitely many $p$ is level-differentially extensible (see \Cref{Multiplier} for details).
\end{remark}

\begin{example}[{\cite{LS}}]\label{detlextl}
Suppose that $\KK$ has characteristic zero.
   The generic determinantal, symmetric determinantal, and skew-symmetric determinantal rings are order-differentially extensible summands of polynomial rings.
\end{example}

In the lemmas that follow, we focus on studying whether a differential operator can be extended to the ring of differential operators of a localization.

\begin{remark} \label{localizationAtElement: L}
   Let $\A$ be an algebra over  $\KK$.
   Then for every element $f$ of $\A$, the localization map $\A\to \A_f$ is differentially extensible with respect to the order filtration.
   Moreover,  as we mentioned in the paragraph preceding \Cref{lemma-logloc},  given $\delta \in D_{\A|\KK}$, there is a unique $\delta_f\in D_{\A_f|\KK}$ that extends $\delta$ \cite[Theorem~2.2.10]{Masson}.
   
   We remark that the usual $D_{\A|\KK}$-module structure on $\A_f$ (see, e.g., \cite[p.~303]{AMHNB}) agrees with the rule $\delta \act v = \delta_f \act v$, where the right-hand side is the tautological $D_{\A_f|\KK}$-action.
\end{remark}

\begin{notation}
   Take $\A$ and $f$ as in \Cref{localizationAtElement: L}.
   Given $\delta \in D_{\A|\KK}$, we retain the notation from the remark, using $\delta_f$ to denote the unique extension of $\delta$ in $D_{\A_f|\KK}$  corresponding to the localization map $\A\to \A_f$.
\end{notation}

\begin{lemma} \label{commutator_extension: L}
   Let  $\varphi\colon \A \to \T$ be a $\KK$-algebra homomorphism, and consider differential operators $\delta, \eta \in D_{\A|\KK}$.
   If $\widetilde{\delta}$ and $\widetilde{\eta}$ are differential operators in $D_{\T|\KK}$ that extend $\delta$ and $\eta$, respectively, then $[\widetilde{\delta}, \widetilde{\eta}] \in D_{\T|\KK}$ extends  $[\delta, \eta] \in D_{\A|\KK}$.
\end{lemma}

\begin{proof}
   Since $\varphi \circ \delta = \widetilde{\delta} \circ \varphi$ and $\varphi \circ \eta = \widetilde{\eta} \circ \varphi$, it is a brief and routine computation to verify that $ \varphi \circ [\delta, \eta] = [\widetilde{\delta}, \widetilde{\eta}] \circ \varphi$.
\end{proof}

\begin{lemma}\label{LemmaExtensionUnique}
Let $\A\subseteq\T$ be $\KK$-algebras.
   For  $f\in \A$, if $\widetilde{\delta}\in D_{\T|\KK}$ extends $\delta\in D_{\A|\KK}$, then $\widetilde{\delta}_f$ extends $\delta_f$.
\end{lemma}

\begin{proof}
   We proceed by induction on the order, $n$, of $\delta$.
   If $n=0$, then $\delta$ is given by multiplication by an element of $\A$, and an extension is given by multiplication by the same element.

   Assume the claim for operators with order less than $n$, and consider $\delta \in D_{\A|\KK}$ with order $n$.
   For a nonnegative integer $t$, set $\rho=[\delta,f^t]\in D_{\A|\KK}$, and note that its order is strictly less than $n$.
   Let $\widetilde{\rho}=[\widetilde{\delta},f^t]\in D_{\T|\KK}$, which by \Cref{commutator_extension: L} is an extension of $\rho$.
   Then for any $h \in \A$, it is straightforward to check that
   \[
      \delta_f\act\frac{h}{f^t}=\frac{\delta \act h-\rho_f\act \frac{h}{f^{t}}}{f^t} \  \text{ and } \ 
      \widetilde{\delta}_f\act\frac{h}{f^t}=\frac{\widetilde{\delta} \act h-\widetilde{\rho}_f \act \frac{h}{f^{t}}}{f^t}.
   \]
   By our inductive hypothesis, $\rho_f \act \frac{h}{f^t}=\widetilde{\rho}_f \act \frac{h}{f^t}$, while $\delta \act h=\widetilde{\delta} \act h$ by assumption; thus, we conclude that $\delta_f \act \frac{h}{f^t} = \widetilde{\delta}_f \act \frac{h}{f^t}$.
\end{proof}

Applying \Cref{LemmaExtensionUnique} to the zero differential operator, we obtain the following.

\begin{corollary}\label{CorZeroLocalization}
  Let $\A \subseteq \T$ be $\KK$-algebras. 
   If $\delta\in D_{\T|\KK}$ vanishes on $\A$, then for every   $f\in \A$, $\delta_f\in D_{\T_f|\KK}$ vanishes on $\A_f$.  \qed
\end{corollary}

\subsection{Differential direct summands and differentially extensible summands}

Let $\A\subseteq \T$ be an extension of Noetherian $\KK$-algebras such that $\A$ is a direct summand of $\T$, with $\A$-linear splitting $\beta\colon \T\to \A$.
A theory of \emph{differential direct summands} was previously developed to study the structure of $D$-modules over $\A$ \cite{AMHNB}.
The definition that we give here is slightly more general than the original, since we also consider $D_{\A|\KK}[\seq{s}]$-modules, though the idea is the same.

\begin{definition}[Differential direct summand {\cite[Definition~3.2]{AMHNB}}]
   Let $\A \subseteq \T$ be an inclusion of $\KK$-algebras with $\A$-linear splitting $\beta\colon \T \to \A$.
   Let $\seq{s}$ be a sequence of indeterminates (which may be empty).
   Recall that, for $\zeta\in D_{\T|\KK}$, the map $\beta \circ \zeta|_\A \colon \A \to \A$ is an element of $D_{\A|\KK}$.
   By abuse of notation, for $\delta\in D_{\T|\KK}[\seq{s}]$, we write $\beta \circ \delta|_\A$ for the element of $D_{\A|\KK}[\seq{s}]$ obtained from $\delta$ by applying $\beta \circ - |_\A$ coefficientwise as a polynomial in $\seq{s}$.

   We say that a $D_{\A|\KK}[\seq{s}]$-module $M$ is a \emph{differential direct summand} of a $D_{\T|\KK}[\seq{s}]$-module $N$ if $M\subseteq N$ and there exists an $\A$-linear splitting $\Theta\colon N\to M$, called a \emph{differential splitting}, such that
   \[  \Theta(\delta \act v) = (\beta \circ \delta|_\A) \act v \]
   for every $\delta\in D_{\T|\KK}[\seq{s}]$ and $v\in M$,
   where the action on the left-hand side is the $D_{\T|\KK}[\seq{s}]$-action, considering $v$ as an element of $N$, and the action on the right-hand side is the $D_{\A|\KK}[\seq{s}]$-action.
\end{definition}

\begin{definition}[Morphism of differential direct summands {\cite[Definition~3.5]{AMHNB}}]
   Let $\A\subseteq \T$ be $\KK$-algebras such that $\A$ is a direct summand of $\T$.
   Fix  $D_{\A|\KK}[\seq{s}]$-modules $M_1$ and $M_2$ that are differential direct summands of $D_{\T|\KK}[\seq{s}]$-modules  $N_1$ and $N_2$, respectively, with differential splittings $\Theta_1\colon N_1\to M_1$ and $\Theta_2\colon N_2\to M_2$.
   We call $\phi\colon N_1\to N_2$ a \emph{morphism of differential direct summands} if $\phi\in \Hom_{D_{\T|\KK}[\seq{s}]}(N_1,N_2)$, $\phi(M_1)\subseteq M_2$, $\phi|_{M_1}\in \Hom_{D_{\A|\KK}[\seq{s}]}(M_1,M_2)$, and the following diagram commutes:
   \[
      \xymatrix{
         M_1 \ar[d]^{\phi|_{M_1}} \ar[r]^{\subseteq} & N_1\ar[d]^{\phi} \ar[r]^{\Theta_1} & M_1\ar[d]^{\phi|_{M_1}}\\
         M_2 \ar[r]^{\subseteq}  & N_2\ar[r]^{\Theta_2} & M_2
      }
   \]
   For simplicity of notation, we often write $\phi$ instead of $\phi|_{M_1}.$

   Further, a complex $M_\bullet$ of $D_{\A|\KK}[\seq{s}]$-modules is called a \emph{differential direct summand} of a complex $N_\bullet$ of $D_{\T|\KK}[\seq{s}]$-modules if each $M_i$ is a differential direct summand of $N_i$, and each differential is a morphism of differential direct summands. 
\end{definition}

\begin{remark}
It is known that the property of being a differential direct summand is preserved under localization, taking kernels, and taking cokernels \cite[Proposition~3.6, Lemma~3.7]{AMHNB}.
\end{remark}

We now move on to develop the notion of \emph{differential extension} in this context, 
which is an essential ingredient in the study of $V$-filtrations and multiplier ideals in later sections.

\begin{definition}[Differentially extensible summand] \label{DefDES}
   Let $\A\subseteq \T$ be a differentially extensible inclusion of $\KK$-algebras, making $\A$ a direct summand of $\T$.
   We say that a $D_{\A|\KK}[\seq{s}]$-module $M$ is a \emph{differentially extensible summand} of a $D_{\T|\KK}[\seq{s}]$-module $N$ (or that $M \subseteq N$ is a \emph{differentially extensible summand}) if the following conditions hold:
   \begin{enumerate}
      \item $M$ is a differential direct summand of $N$.
      \item For any $\delta\in D_{\A|\KK}[\seq{s}]$ and any $\widetilde{\delta}\in D_{\T|\KK}[\seq{s}]$ that extends $\delta$, one has  $\widetilde{\delta}\act m=\delta \act m$ for all $m\in M$.
   \end{enumerate}
\end{definition}

Now we  study the properties of differentially extensible summands.
In particular, we see that the operations of localization, taking kernels, and taking cokernels all preserve the property of being a differentially extensible summand. 
This is important, due to our desire to work with local cohomology, which we use to investigate $V$-filtrations.

\begin{lemma}\label{LemmaLoc}
 Let $\A\subseteq \T$ be a differentially extensible inclusion of $\KK$-algebras, with $\A$ a direct summand of $\T$.
   If a  $D_{\A|\KK}$-module $M$ is a differentially extensible summand of a $D_{\T|\KK}$-module $N$, then for any $f\in \A$, the localization $M_f$ is a differentially extensible summand of $N_f$.
   Furthermore, the localization map is a morphism of differential direct summands.
\end{lemma}

\begin{proof}
   It has already been established that $M_f$ is a differential direct summand of $N_f$ \cite[Proposition~3.6]{AMHNB}.
   Given an extension $\widetilde{\delta}\in D_{\T|\KK}$ of $\delta \in D_{\A|\KK}$, we have that $\widetilde{\delta} \act  \frac{v}{f^t} = \delta \act \frac{v}{f^t}$ for all $v \in M$,  and all integers $t \geq 0$, by an argument directly analogous to the proof of \Cref{LemmaExtensionUnique}, but with $v \in M$ replacing $h\in \A$. 
\end{proof}

\begin{lemma}\label{LemmaKerImCoKer}
   Let $\A\subseteq \T$ be a differentially extensible inclusion of $\KK$-algebras,  with $\A$ a direct summand of $\T$.
   Consider $D_{\A|\KK}$-modules $M_1$ and $M_2$ that are differentially extensible summands of $D_{\T|\KK}$-modules  $N_1$ and $N_2$.
   If $\phi\colon N_1\to N_2$ is a morphism of differential direct summands, then $\Ker(\phi|_{M_1})\subseteq \Ker(\phi)$, $\IM(\phi|_{M_1})\subseteq \IM(\phi )$, and $\CoKer(\phi|_{M_1})\subseteq \CoKer(\phi)$ are also differentially extensible summands.
\end{lemma}

\begin{proof}
   It is known that $\Ker(\phi|_{M_1})\subseteq \Ker(\phi)$, $\IM(\phi|_{M_1})\subseteq \IM(\phi )$, and $\CoKer(\phi|_{M_1})\subseteq \CoKer(\phi)$, and that, moreover, each is a differential direct  summand \cite[Lemma~3.7]{AMHNB}.

   Take an extension $\widetilde{\delta}\in D_{\T|\KK}$ of $\delta\in D_{\A|\KK}$.
   Since $\widetilde{\delta} \act v = \delta \act v$ for all $v \in M_1$, this holds for all $v \in \ker(\phi|_{M_1}) \subseteq M_1$, and $\widetilde{\delta} \act u = \delta \act u$ for all $u\in  \operatorname{Coker}(\phi|_{M_1})$.
   Moreover, since  $\delta \act w=\widetilde{\delta} \act w$ for all $w \in M_2$, the same holds for all $w \in \IM(\phi|_{M_1}) \subseteq M_2$.
\end{proof}

\begin{proposition}\label{PropLcExt}
   Let $\A\subseteq \T$ be a differentially extensible inclusion of $\KK$-algebras, with $\A$ a direct summand of $\T$.
   Let  $M$ be a $D_{\A|\KK}$-module, and $N$  a $D_{\T|\KK}$-module, with $M$ a differentially extensible summand of $N$.
   Then for any ideal $I$ of $\A$, and any integer $i$, $H^i_I(M)$ is a differentially extensible summand of $H^i_{I\T}(N)$. 
\end{proposition}

\begin{proof}
   Fix generators $\seq{f}=\seq[\ell]{f} \in \A$ for $I$.
   The \v{C}ech-like complex $\Cech^\bullet(\seq{f};M)$ defining the local cohomology modules of $M$ with support in $I$ is a differentially extensible summand of the complex $\Cech^\bullet(\seq{f};N)$ defining the local cohomology of $N$ with support in $I\T$:
   It is apparent that $\Cech^j(\seq{f};M) \subseteq \Cech^j(\seq{f};N)$ for all integers $j$, and each is a differentially extensible summand by  \Cref{LemmaLoc}; moreover, the fact that each differential is a morphism of differentially extensible summand holds by this same lemma, since each localization is such a morphism.
   Therefore, by \Cref{LemmaKerImCoKer}, the inclusion of cokernels $H^i_I(M)\subseteq H^i_{I\T}(N)$ is a differentially extensible summand.
\end{proof}

 We now study the notions of differentiable direct summands and differentially extensible summands in the context of the modules $M^\A[\fsl]$ introduced in \Cref{ss: Bernstein--Sato for nonregular rings}.

\begin{setup}\label{setup: diff ext}
   In the remainder of this subsection, $\KK$ is a field of characteristic zero, and $\A\subseteq \T$ are Noetherian
   $\KK$-algebras such that $\A$ is a direct summand of $\T$. 
   We fix elements $\seq{f}=\seq[\ell]{f}\in \A$, and assume that $M^\A[\fsl]$ and $M^\T[\fsl]$ have $D$-module structures compatible with specialization, as in \Cref{wellDefinedActionOnMR: P}.
   Recall that such a structure, if it exists, is unique, by the argument used in the proof of \Cref{wellDefinedActionOnMR: P}; moreover, that structure does exist for
   $\KK$-algebras that are either finitely generated or complete.
\end{setup}

\begin{theorem} \label{MRfs_dds}
   In the context of \Cref{setup: diff ext}, if $\beta\colon \T\to \A$ is an $\A$-linear splitting, the $D_{\A|\KK}[\seq{s}]$-module $M^\A[\fsl]$ is a differential direct summand of the $D_{\T|\KK}[\seq{s}]$-module $M^{\T}[\fsl]$, with differential splitting
   \[
      \Theta =  M^\A[\fsl]\otimes_{\A} \beta\colon M^\T[\fsl] \to M^\A[\fsl].
   \]
\end{theorem}

\begin{proof}
 Given $\eta\in D_{\T|\KK}[\seq{s}]$ we consider the differential operator  $\delta\coloneqq  \beta \circ \eta|_\A \in D_{\A|\KK}[\seq{s}]$.  Then, for any 
 $v\in M^\A[\fsl]$ we have to check that $\delta \act v = \Theta(\eta \act v)$. By \Cref{lemma-evaluation}, it suffices to show that $\varphi_{\pt{t}}(\Theta(\eta \act v)) = \varphi_{\pt{t}}(\delta \act v)$ for all $\pt{t}\in \ZZ^{\ell}$.
	Indeed, for each such $\pt{t}$ we have
	\[
	\varphi_{\pt{t}}(\Theta(\eta \act v)) = \beta_{f_1\cdots f_{\ell}} ( \varphi_{\pt{t}} (\eta \act v))
	= (\beta \circ \eta(\pt{t}) |_\A) \act \varphi_{\pt{t}}(v) = \delta(\pt{t}) \act \varphi_{\pt{t}}(v) = \varphi_{\pt{t}}(\delta \act v),
	\]
	where we have used, in turn, the definition of $\Theta$, the fact that $\A_{f_1 \cdots f_{\ell}}$ is a differential direct summand of $\T_{f_1 \cdots f_{\ell}}$ with differential splitting $\beta_{f_1\cdots f_{\ell}}$ \cite[Proposition~3.6]{AMHNB} and the definition of $\delta$.
\end{proof}




\begin{theorem} \label{MRfs_des}
   Under \Cref{setup: diff ext}, if $\A\subseteq \T$ is a differentially extensible inclusion, then the $D_{\A|\KK}[\seq{s}]$-module $M^\A[\fsl]$ is a differentially extensible summand of the $D_{\T|\KK}[\seq{s}]$-module $M^{\T}[\fsl]$.
\end{theorem}

\begin{proof}
   We proved in \Cref{MRfs_dds} that $M^\A[\fsl]$ is a differential direct summand of $M^{\T}[\fsl]$.
   To prove the second condition of differential extensibility, fix $v\in M^\A[\fsl]$ and $\delta\in D_{\A|\KK}[\seq{s}]$, and suppose $\widetilde{\delta} \in D_{\T|\KK}[\seq{s}]$ extends $\delta$.
   Write $\delta \act v = g \fsll$ and $\widetilde{\delta}\act v = h \fsll$, with $g\in \A_{f_1\cdots f_\ell}[\seq{s}]$ and $h\in \T_{f_1\cdots f_\ell}[\seq{s}]$.
   Since the $D$-module structures on $M^\A[\fsl]$ and $M^\T[\fsl]$ are compatible with specialization, for each $\pt{t}\in \ZZ^\ell$ we have
   \[
      g(\pt{t})f_1^{t_1} \cdots f_{\ell}^{t_\ell} = \varphi_{\pt{t}}(\delta \act v) = \delta(\pt{t})\act \varphi_{\pt{t}}(v) = \widetilde{\delta}(\pt{t})\act \varphi_{\pt{t}}(v) = \varphi_{\pt{t}}(\widetilde{\delta} \act v) = h(\pt{t})f_1^{t_1} \cdots f_{\ell}^{t_\ell},
   \]
   where the central equality follows from \Cref{LemmaExtensionUnique}.
   This shows that $g(\pt{t})=h(\pt{t})$, for all $\pt{t}\in \ZZ^\ell$, so $g=h$ by \Cref{lemma-evaluation}, and $\delta \act v = \widetilde{\delta} \act v$.
\end{proof}

\subsection{Bernstein--Sato functional equation for direct summands} \label{BS_des}

As we have done in \Cref{Exist_BS}, the existence of the formal Bernstein--Sato polynomial for direct summands can be proved by invoking the existence of the specialized version \cite{AMHNB}.
However we can provide a direct proof using the $D$-module structure of the module $M^\A[\fsl]$ described in \Cref{wellDefinedActionOnMR: P} and following the ideas used for the specialized version \cite{AMHNB}.
For completeness we work out the details in this section in full generality.
Moreover, we  address
the issue of whether the Bernstein--Sato polynomial associated to a sequence of elements  only depends on the ideal generated by these elements. 

From now on we consider the following setup but we point out that all the results also hold for direct summands of formal power series rings.

\begin{setup}\label{setup}
   In the remainder of this section, $\KK$ is a field of characteristic zero, $\R$ is a polynomial ring over $\KK$, and $\A$ is a $\KK$-subalgebra of $\R$, such that $\A$ is a direct summand of $\R$ and  that $\A$ is finitely generated over $\KK$.
   Moreover, following \Cref{convention: D-module structure}, given $\seq{f}=\seq[\ell]{f}\in \A$, when referring to the $D_{\A|\KK}[\seq{s}]$-module structure on $M^\A[\fsl]$, we use the structure described in \Cref{wellDefinedActionOnMR: P}.
\end{setup}

For the sake of clarity, we first present the more familiar case of principal ideals. 

\begin{theorem}\label{ThmExistenceBS}
   Under \Cref{setup}, given $f,g\in \A$, there exists $\delta\in D_{\A|\KK}[s]$, and a nonzero polynomial $b(s)\in \QQ[s]$, such that the functional equation
   \[
      \delta \act f g \fs= b(s)g \fs
   \]
   holds in $M^\A[\fs]$.  
\end{theorem} 

\begin{proof}
  There exists ${\eta}\in D_{\R|\KK}[s]$ and a nonzero polynomial $b(s)\in \QQ[s]$ such that
   \[
      \eta \act  f g \fs= b(s) g\fs
   \]
   in $M^\R[\fs]$, as in \Cref{bspairpoly}.
   By applying the splitting $\Theta\colon M^\R[\fs]\to M^\A[\fs]$ induced by the splitting $\beta$ making $\A$ a direct summand of $\R$, we obtain that
   \[
      b(s)g\fs= \Theta( b(s)g\fs) = \Theta(\eta\act f g \fs)=(\beta\circ {\eta}|_\A) \act f  g \fs
   \]
   in  $M^\A[\fs]$. Taking $\delta=\beta\circ {\eta}|_\A \in D_{\A | \KK }[\seq{s}]$, we obtain the desired equation.
\end{proof} 	

\begin{example}
   Let $\displaystyle \A=\frac{\CC[u,v,w]}{\ideal{u^3-vw}}$.
   We  determine a functional equation for the element $uv\in \A$.
   To do this, we realize $\A \subseteq \R= \CC[x,y]$ via the inclusion given by $f(u)=xy$, $f(v)=x^3$, $f(w)=y^3$.
   Observe that the image of $\A$ under $f$ is the ring of polynomials invariant under the action of the cyclic group $G$ generated by $g$, where $g(x)=e^{2\pi i/3} x$ and $g(y)= e^{-2\pi i/3} y$.
   Since $g$ does not fix any one-forms in $\R$, the inclusion map $f$ is differentially extensible by \Cref{E-finitegroup}.
   From here we find a functional equation for $f(uv)=x^4 y$ in $\R$, namely
   \[
      \frac{1}{4^4}\,\pd{x}^4\pd{y}[x^4 y \boldsymbol{(x^4 y)^s}] =
      (s+1)^2\left(s+\frac{3}{4}\right)\left(s+\frac{1}{2}\right)\left(s+\frac{1}{4}\right) \boldsymbol{(x^4 y)^s}.
   \]
   The existence of such an operator comes from the standard theory of Bernstein--Sato polynomials.
   We then observe that the differential operator $\frac{1}{4^4}\,\pd{x}^4\pd{y} \in D_{\R | \CC}[s]$ sends $G$-invariant polynomials to $G$-invariant polynomials, and is thus a differential operator on $\A$.
   This yields the functional equation we seek. 
\end{example}

Thanks to \Cref{ThmExistenceBS}, the following definition is well founded for a direct summand of a polynomial 
ring. 

\begin{definition}[Relative Bernstein--Sato polynomial of an element of a direct summand of a polynomial 
ring]
Given $\A \subseteq \R$ and $f, g \in \A$ as in \Cref{ThmExistenceBS}, we call the monic polynomial of smallest degree satisfying the conclusion of this theorem the \emph{Bernstein--Sato polynomial of $f$ relative to $g$}, which we denote $b_{f,g}^\A(s)$.
   If $g=1$, we call this polynomial the \emph{Bernstein--Sato polynomial of $f$}, and denote it $b^\A_f(s)$.
\end{definition}

We now present the Bernstein--Sato polynomial $b_{\seq{f},g}^\A(s)$  associated to a set of nonzero generators $\seq{f}=\seq[\ell]{f}$ of an ideal $I$ of $\A$, relative to an element $g \in \A$.

\begin{theorem}\label{ThmBSVarieties}
   Under \Cref{setup}, given $\seq{f}=\seq[\ell]{f} \in \A$ and $g\in \A$, there exists a nonzero polynomial $b(s)\in \QQ[s]$ such that
   \[
      b(s_1+\cdots +s_\ell)  g\fsll
   \]
   is in the $D_{\A|\KK}[\seq{s}]$-submodule of $M^\A[\fsl]$ generated by all
   \[
      \prod_{\crampedclap{i\in\nsupp(\pt{c})}} \quad \binom{s_i}{-c_i} f^{c_1}_1\cdots f^{c_\ell}_\ell  g\fsll,
   \]
   where $\pt{c}=\pt[\ell]{c}$ runs over the elements of $ \ZZ^{\ell}$ such that $|\pt{c}|=c_1+\cdots+c_\ell=1$, and $\nsupp(\pt{c})=\{i:c_i<0\}$. 
\end{theorem} 

\begin{proof}
   Let $N$ denote the $D_{\A|\KK}[\seq{s}]$-submodule of $M^\A[\fsl]$ described in the statement of the theorem.
   By the existence results given for Bernstein--Sato polynomials for varieties in smooth varieties \cite[Section~2.10]{BMS2006a} (see also \cite[Definition~2.2]{BudurNotes}), there exists a nonzero polynomial $b(s)\in \QQ[s]$ such that $b(s_1+\cdots +s_\ell)  g\fsll$ is in the $D_{\R|\KK}[\seq{s}]$-submodule of $M^\R[\fsl]$ generated by all
   \[
      \prod_{\crampedclap{i\in\nsupp(\pt{c})}}\quad \binom{s_i}{-c_i} f^{c_1}_1\cdots f^{c_\ell}_\ell  g\fsll,
   \]
   where $\pt{c}=\pt[\ell]{c}$ runs over the elements of $ \ZZ^{\ell}$ such that $|\pt{c}|=1$.
   Now, applying the splitting $\Theta\colon M^\R[\fsl]\to M^\A[\fsl]$, we obtain that
   \[
      b(s_1+\cdots +s_\ell)  g\fsll \in N.
      \qedhere
   \]
\end{proof} 
 
As in the case of a single element $f$, we have justified the following definition in the case of a direct summand of a polynomial  
ring over a field of characteristic zero.
	
\begin{definition}[Relative Bernstein--Sato polynomial of a sequence of elements in a direct summand of a polynomial 
ring]
   We call the unique monic polynomial $b(s) \in \QQ[s]$ of smallest degree satisfying the conclusion of \Cref{ThmBSVarieties} the \emph{Bernstein--Sato polynomial of $\seq{f}$ relative to $g$}, and we denote it $b_{\seq{f},g}^\A(s)$.
   If $g=1$, we call this the \emph{Bernstein--Sato polynomial of $\seq{f}$}, and we denote it $b^\A_{\seq{f}}(s)$.
\end{definition}

It follows from the proof of \Cref{ThmBSVarieties} that, 
for a direct summand $\A\subseteq \R$, the Bernstein--Sato polynomial $b^\A_{\seq{f}, g}(s)$   divides  $b^\R_{\seq{f}, g}(s)$  (see \cite[Theorem~3.14]{AMHNB} for the specialized version).
Under the extra condition that $\A\subseteq \R$ is differentially extensible, these two polynomials coincide as it was shown for the specialized version \cite[Theorem~6.11]{BJNB}. 

\begin{theorem}\label{ThmBSEqual}
 Consider \Cref{setup} under the extra condition that the inclusion $\A\subseteq \R$ is differentially extensible.  Given $\seq{f}=\seq[\ell]{f} \in \A$ and $g\in \A$, we have that $b^\A_{\seq{f}, g}(s)=b^\R_{\seq{f}, g}(s).$
\end{theorem}

\begin{proof}
   From the proof of \Cref{ThmBSVarieties}, we obtain that $b^\A_{\seq{f}, g}(s)$ divides $b^\R_{\seq{f}, g}(s)$.
   We now prove that $b^\R_{\seq{f}, g}(s)$ divides $b^\A_{\seq{f}, g}(s)$.
   Choose $\delta_{\pt{c}} \in D_{\A|\KK}[\seq{s}]$ for which
   \[
      b^\A_{\seq{f}, g}(s_1+\cdots+ s_\ell)  g\fsll =
      \sum_{\crampedclap{\substack{\pt{c} \in \ZZ^\ell \\ |\pt{c}|=1} }} \delta_{\pt{c}}\act\ \   \prod_{\crampedclap{i\in\nsupp(\pt{c})}}
      \quad \binom{s_i}{-c_i}  f^{c_1}_1\cdots f^{c_\ell}_\ell g \fsll.
   \]
   The above equation remains true after replacing each $\delta_{\pt{c}} \in D_{\A|\KK}[\seq{s}]$ with an extension $\widetilde{\delta}_{\pt{c}} \in D_{\R|\KK}[\seq{s}]$, justifying our claim.
\end{proof}




Using a mild generalization of a recent result by Musta{\c{t}}{\u{a}} \cite{Mustata2019} we can prove that the Bernstein--Sato polynomial does not depend on the choice of generators for an ideal $I\subseteq \A$, justifying the notation $b^\A_{I, g}(s)$, or $b^\A_I(s)$ if $g=1$. 

\begin{theorem}[Musta{\c{t}}{\u{a}}]\label{ThmExtMustata}
Let $A$ be an algebra over a field $\KK$ of characteristic zero, that is finitely generated over $\KK$.
Assume that $\A$  is a direct summand of a polynomial  ring over $\KK$. Given nonzero $\seq{f} = \seq[\ell]{f}\in A$ and $g\in A$, set $h = y_1f_1+\cdots + y_\ell f_\ell \in A[y_1,\dots, y_\ell]$, where $y_1,\dots, y_\ell$ are  indeterminates. Then \[b^A_{\seq{f},g}(s)=\frac{{ b}^{A[y_1,\dots, y_\ell]}_{h,g}(s) }{(s+1)}.\]
Furthermore, $b^A_{\seq{f},g}(s)$ only depends on the ideal generated by $\seq{f}$.
\end{theorem}

The proof of this result is essentially the same as the one presented in work previously mentioned \cite[Theorem~1.1 and Remark~2.1]{Mustata2019}, so we omit it.
We only note that the hypothesis that $A$ (and hence also $A[y_1,\dots, y_\ell]$) is a direct summand of a polynomial ring guarantees the existence of the relative Bernstein--Sato polynomial ${ b}^{A[y_1,\dots, y_\ell]}_{h,g}(s)$, and that $(s+1)$ is a factor of this polynomial, since $h$ involves some variables that $g$ does not. 

\Cref{ThmBSVarieties,ThmExtMustata} allow us to define the (relative) Bernstein--Sato polynomial of an ideal in a direct summand of  a polynomial ring.

\begin{definition}[Relative Bernstein--Sato polynomial in a direct summand of a polynomial 
ring]
Let $A$ be an algebra over a field $\KK$ of characteristic zero that is finitely generated over $\KK$.
Assume that $\A$  is a direct summand of a polynomial  ring over $\KK$.
   Let $I$ and $g$ be an ideal and an element of $\A$.
   Given generators $\seq{f} = \seq[\ell]{f}$ for $I$, we call the polynomial $b^\A_{\seq{f},g} \in \QQ[s]$ the   
   \emph{Bernstein--Sato polynomial of $I$ relative to $g$}, and denote it $b^\A_{I,g}(s)$.
   If $g=1$, we call this the \emph{Bernstein--Sato polynomial of $I$}, and denote it $b^\A_I(s)$.
\end{definition}

\section{$V$-filtrations} \label{Vfilt}

The theory of the \emph{$V$-filtration}, originally defined by Malgrange \cite{MalgrangeVfil} and Kashiwara \cite{KashiwaraVfil}, is a central object in the study of $D$-modules over a smooth variety.
Our goal in this section is to extend this theory to the setting of a differentially extensible summand of a polynomial ring.
In the next section, we  apply this newly-developed theory to the study of multiplier ideals.

Suppose that $\KK$ has characteristic zero.
Let  $\R=\KK[\seq{x}]=\KK[\seq[d]{x}]$ be a polynomial ring over $\KK$  and set $X=\Spec(\R)$.
Let $I\subseteq \R$ be an ideal generated by $\seq{f}=\seq[\ell]{f} \in \R$.
We may assume that $I$ defines a smooth subvariety,  by considering the graph embedding $i_{\seq{f}} \colon X \to X \times \mathbb{A}^\ell$ mapping $x \mapsto (x, f_1(x), \dots , f_\ell(x))$, and then replacing $X$ by $ X \times \mathbb{A}^\ell$ and each $f_i$ by the projection onto the $i$th coordinate of $\mathbb{A}^\ell$.
Therefore, if $\seq{t}=\seq[\ell]{t}$ is a system of coordinates of $\mathbb{A}^\ell$,  we may simply assume that $I$ is  $\ideal{\seq{t}}=\ideal{\seq[\ell]{t}} \subseteq \R[\seq[\ell]{t}]$. 

The $V$-filtration along $\ideal{\seq{t}}$ on the ring of $\KK$-linear differential operators  $D_{\R[\seq{t}]|\KK}$ is obtained by assigning degree $0$ to the elements of $D_{\R|\KK}$, degree $1$ to the indeterminates $\seq[\ell]{t}$, and degree $-1$ to the partial derivatives $\pd{t_1},\ldots,\pd{t_\ell}$.
Equivalently, a differential operator $\delta$ lies in $ V^i D_{\R[\seq{t}]|\KK}$ if and only if $\delta \act \ideal{\seq{t}}^j \subseteq  \ideal{\seq{t}}^{j+i} \ \text{for all} \ j \in \ZZ$, where we use the convention that $\ideal{\seq{t}}^j = \R[\seq{t}]$ for negative exponents $j$.

For a given $D_{\R|\KK}$-module $N$, one considers its direct image with respect to the graph embedding
\[{i_{\seq{f}}}_\ast(N)\coloneqq  N\otimes_{\KK} \KK[\pd{t_1}, \dots , \pd{t_\ell}],\]
which is a $D_{\R[\seq{t}]|\KK}$-module isomorphic to the local cohomology module
\[H^\ell\wrttminusf(N[\seq{t}]),\]
where $\idtminusf$ denotes the ideal $\ideal{t_1-f_1, \dots , t_\ell -f_\ell}$ of $\R[\seq{t}]$.
If $N$ is a \emph{regular holonomic} $D_{\R|\KK}$-module with \emph{quasi-unipotent monodromy}, then there exists a $V$-filtration along $\ideal{\seq{t}}$ on ${i_{\seq{f}}}_\ast(N)$, which is a decreasing filtration $V^\alpha\wrtt {i_{\seq{f}}}_\ast(N)$ indexed over the rational numbers satisfying certain conditions (see  \Cref{VfilM}).
The existence of such a filtration is a result of Malgrange \cite{MalgrangeVfil} and Kashiwara \cite{KashiwaraVfil}, and is essentially equivalent to the existence of relative Bernstein--Sato polynomials in a generalized sense.  We do not go into detail on the definition of regular holonomic modules with quasi-unipotent monodromy, so we encourage the interested reader to take a look at \emph{loc.\,cit.}; we only point out that the ring $\R$ itself, and the local cohomology modules of the form $H^i_I(\R)$, satisfy this property.

Finally, one defines the $V$-filtration along $\seq{f}$ of the $D_{\R|\KK}$-module $N$ as
\[V^\alpha\wrtf N \coloneqq V^\alpha\wrtt {i_{\seq{f}}}_\ast(N) \cap N.\]
It is known  that the $V$-filtration only depends on the ideal generated by $\seq{f}$, and not on the choice of generators \cite[Proposition~2.6]{BMS2006a}.

\subsection{$V$-filtrations for nonregular rings} \label{Vfilt_non_reg}

In order to develop a theory of $V$-filtrations for a not-necessarily-regular ring, we follow the same ideas as in the smooth case and define them using a set of axioms.


\begin{definition}[$V$-filtration on $D_{\T[\seq{t}]|\KK}$] \label{DefVfil}
Suppose that $\KK$ has  characteristic zero. Let $\T$ be a Noetherian commutative   $\KK$-algebra,  and let $\T[\seq{t}]=\T[\seq[\ell]{t}]$ be a polynomial ring over $\T$.  The $V$-filtration along the ideal $\ideal{\seq{t}}=\ideal{\seq[\ell]{t}}$ on the ring of differential operators  $D_{\T[\seq{t}]|\KK}$ is the filtration indexed by integers $i\in \mathbb{Z}$ defined by
   \begin{equation*}
   V^{i}\wrtt {D_{\T[\seq{t}]|\KK}} = \{ \delta \in D_{\T[\seq{t}]|\KK} :  \delta \act \ideal{\seq{t}}^j \subseteq  \ideal{\seq{t}}^{j+i} \ \text{for all} \ j \in \ZZ \},
\end{equation*}
   where we use the convention that $\ideal{\seq{t}}^j = \T[\seq{t}]$ for integers $j \leq 0$.
\end{definition}

\begin{remark}
   As a graded $D_{\T|\KK}$-module, we have  the following description:
   \begin{equation*}
      V^{i}\wrtt {D_{\T[\seq{t}]|\KK}}
      = \bigoplus_{\substack{\pt{a},\pt{b} \in \NN^\ell \\ |\pt{a}|-|\pt{b}|\,\geq \, i}} D_{\T|\KK} \cdot  t_1^{a_1}\cdots t_\ell^{a_\ell} \pd{t_1}^{b_1} \cdots \pd{t_\ell}^{b_\ell},
   \end{equation*}
   where, given $\pt{u} = \pt[\ell]{u} \in \NN^\ell$, we denote $|\pt{u}| = u_1 + \cdots + u_\ell$.
\end{remark}

The $V$-filtration along the ideal $\ideal{\seq{t}}=\ideal{\seq[\ell]{t}}$ on a $D_{\T[\seq{t}]|\KK}$-module $M$ is defined axiomatically as follows.

\begin{definition}[$V$-filtration on a $D_{\T[\seq{t}]|\KK}$-module along indeterminates] \label{VfilM}
   Let $M$ be a $D_{\T[\seq{t}]|\KK}$-module.
   A \emph{$V$-filtration on $M$ along the ideal $\ideal{\seq{t}}=\ideal{\seq[\ell]{t}}$} is a decreasing filtration $\{V^{\alpha}\wrtt M\}_\alpha$ on $M$, indexed by $\alpha\in \QQ$, satisfying the following conditions.
   \begin{enumerate}
      \item For all $\alpha\in \QQ$, $V^{\alpha}\wrtt M$ is a Noetherian  $V^{0}\wrtt {D_{\T[\seq{t}]|\KK}}$-submodule of $M$.
      \item The union of the $V^{\alpha}\wrtt M$, over all $\alpha \in \QQ$, is $M$.
      \item  $V^{\alpha}\wrtt M = \bigcap_{\gamma < \alpha} V^{\gamma}\wrtt M$ for all $\alpha$, and the set $J$  consisting of all $\alpha\in \QQ$ for which $V^{\alpha}\wrtt M \neq \bigcup_{\gamma > \alpha} V^{\gamma}\wrtt M$ is discrete.
      \item For all $\alpha \in \QQ$ and all $1 \leq i \leq \ell$,
      \[ t_i \act V^\alpha\wrtt M \subseteq V^{\alpha+1}\wrtt M \, \text{ and } \,\pd{t_i} \act V^\alpha\wrtt M \subseteq V^{\alpha-1}\wrtt M, \] i.e., the filtration is compatible with the $V$-filtration on $D_{\T[\seq{t}]|\KK}$.
      \item For all $\alpha\gg 0$,  $\sum^\ell_{i=1}\left( t_i \act V^\alpha\wrtt M\right) = V^{\alpha+1}\wrtt M$.
      \item For all $\alpha \in \QQ$,
      \[\sum_{i=1}^\ell \pd{t_i} t_i - \alpha\]
      acts nilpotently on $V^\alpha\wrtt M / (\bigcup_{\gamma > \alpha} V^{\gamma}\wrtt M)$.
   \end{enumerate}
\end{definition}
		
\begin{remark}
   Let $J$ be the set 
   introduced in Condition~(3) above.
   It follows easily from Condition~(5) that for $\alpha\gg 0$, $\alpha+1 \in J$ implies that $\alpha \in J$.
   By Condition~(1), the set $J$ is bounded below.
   Thus, given the fact that these two conditions hold, Condition~(3) is equivalent to the existence of some $m\in \NN$ such that for every $n\in \ZZ$ and all $\alpha\in \left( \frac{n-1}{m}, \frac{n}{m} \right]$, $V^\alpha\wrtt M = V^{n/m}\wrtt M.$
\end{remark}

As in the case of polynomial rings, the $V$-filtration is unique in this broader context, provided it exists. 
The proof of this fact is analogous to that of the original result in the case that $\T$ is a polynomial ring over $\KK$ (see \cite[Proposition~1.3]{SurveyBudur}), so we omit it. 
	
\begin{proposition}
   Let $M$ be a finitely generated $D_{\T[\seq{t}]|\KK}$-module. If a $V$-filtration on $M$ along $\ideal{\seq{t}}$ exists, then it is unique.
   \qed
\end{proposition}

Our goal is  to prove the existence of the $V$-filtration for differentially extensible summands of polynomial rings.
For this reason, we introduce certain auxiliary Rees algebras that lead to an equivalent way of defining $V$-filtrations.

\begin{definition} \label{Rees}
Suppose that $\KK$ has  characteristic zero. Let $\T$ be a Noetherian commutative   $\KK$-algebra,  and let $\T[\seq{t}]=\T[\seq[\ell]{t}]$ be a polynomial ring over $\T$.  Given an auxiliary indeterminate $w$, we define the following Rees algebras:
\begin{alignat*}{2}
	&\VV\wrtt D_{\T[\seq{t}]|\KK} &&\coloneqq\bigoplus_{i\in \ZZ} \left( V^i\wrtt D_{\T[\seq{t}]|\KK} \right) w^i \\ 
	&&&= D_{\T|\KK}[ \,t_1 w, \dots , t_\ell w, \pd{t_1} w^{-1}, \dots ,  \pd{t_\ell} w^{-1}] \\
        &\VV^+\wrtt D_{\T[\seq{t}]|\KK} &&\coloneqq  \bigoplus_{i\ge 0} \left( V^i\wrtt D_{\T[\seq{t}]|\KK} \right) w^i\  \oplus\  \bigoplus_{i<0} \left( V^0\wrtt D_{\T[\seq{t}]|\KK}\right) w^i \\ &&&=
        D_{\T|\KK}[  w^{-1}, t_i w, \partial_{t_i} t_j \mid 1 \leq i, j \leq \ell]
\end{alignat*}
Moreover, given a $D_{\T[\seq{t}]|\KK}$-module $M$, for $\alpha \in \QQ$, 
\[
\VV^\alpha\wrtt M \coloneqq \bigoplus_{i\in \ZZ} \left( V^{\alpha+i}\wrtt M \right) w^i.
\]
\end{definition}

In the polynomial ring case, these algebras are Noetherian.
 
\begin{lemma} \label{LemmaVAlgebrasNoeth}
   Suppose that $\KK$ has characteristic zero.
   Let $\R=\KK[\seq{x}]=\KK[\seq[d]{x}]$, and consider the polynomial ring $\R[\seq{t}]=\R[\seq[\ell]{t}]$ over $\R$.
   Then $V^0\wrtt D_{\R[\seq{t}]|\KK}$, $\VV\wrtt D_{\R[\seq{t}]|\KK}$, and $\VV^+\wrtt D_{\R[\seq{t}]|\KK}$ are Noetherian rings.
\end{lemma}

\begin{proof}
   We aim to show that each ring admits a filtration for which the associated graded ring is commutative and Noetherian.
   Let $\mathcal{F}$ denote the order filtration on $D_{\R[\seq{t}]|\KK}$, so that
   \[\gr_\mathcal{F}(D_{\R[\seq{t}]|\KK})\cong \KK[\seq{x},\seq{y},\seq{t},\seq{u}],\]
   where $\seq{y}=\seq[d]{y}$ and $\seq{u}=\seq[\ell]{u}$ correspond to the images of the partial derivatives $\pd{x_1}, \ldots, \pd{x_d}$ and $\pd{t_1},\ldots,\pd{t_\ell}$. 
   Let $\mathcal{G}_n = \left( V^{0}\wrtt {D_{\R[\seq{t}]|\KK}} \right) \cap \mathcal{F}_n$, and consider the ring inclusion
   \[\varphi\colon \gr_\mathcal{G}(V^{0}\wrtt {D_{\R[\seq{t}]|\KK}})\to \gr_\mathcal{F}(D_{\R[\seq{t}]|\KK}).\]
   Set $\pt{a}=\pt[\ell]{a}$ and $\pt{b}=\pt[\ell]{b}$.
   We have that 
	\begin{align*}
	\gr_\mathcal{G}(V^{0}\wrtt {D_{\R[\seq{t}]|\KK}}) \cong \IM(\varphi)& \cong  \KK[\seq{x},\seq{y}, t_1^{a_1} \cdots t_{\ell}^{a_\ell} u_1^{b_1} \cdots u_{\ell}^{b_\ell}]_{ \{|\pt{a}| \geq |\pt{b}|\} }\\
	&\cong \KK[\seq{x},\seq{y},\seq{t}, t_i u_j]_{\{i,j\}},
	\end{align*}
	which is a finitely generated $\KK$-algebra, and so, a Noetherian ring.
        We conclude that  $V^{0}\wrtt {D_{\R[\seq{t}]|\KK}}$ is left- and right-Noetherian since $\mathcal{F}_0\cong \R[ \seq{t}]$.
	
	Now, consider the subring $\VV\wrtt D_{\R[\seq{t}]|\KK}$ of $D_{\R[\seq{t}]|\KK}[w]$, which is $\ZZ$-graded as a polynomial ring in $w$. We refine this to a $\ZZ\times \NN$ grading by taking the filtration
        \[\mathcal{G}'_{j}=\{ \delta \in \VV\wrtt D_{\R[\seq{t}]|\KK} : \delta\in \mathcal{F}_j\otimes_{D_{\R[\seq{t}]|\KK}} D_{\R[\seq{t}]|\KK}[w] \}.\]
	We obtain an inclusion $\varphi\colon \gr_{\mathcal{G}'}(\VV\wrtt D_{\R[\seq{t}]|\KK})\to \gr_\mathcal{F}(D_{\R[\seq{t}]|\KK})[w]$.
        With the same coordinates that we used for $\gr_\mathcal{F}(D_{\R[\seq{t}]|\KK})$ above, 
	 	\begin{align*}
	 \gr_{\mathcal{G}'}(\VV\wrtt D_{\R[\seq{t}]|\KK}) \cong \IM(\varphi)& \cong  \KK[\seq{x},\seq{y}, t_1^{a_1} \cdots t_{\ell}^{a_\ell} u_1^{b_1} \cdots u_{\ell}^{b_\ell} w^j]_{\{|a| \geq |b| + j\} }\\
	 &\cong \KK[\seq{x},\seq{y},t_1w,\ldots, t_\ell w, u_1 w^{-1},\dots, u_\ell w^{-1}],
	 \end{align*}
	 which is again a finitely generated $\KK$-algebra, and hence Noetherian. We conclude that $\VV\wrtt D_{\R[\seq{t}]|\KK}$ is left- and right-Noetherian.
	 
	 Finally, consider $\VV^+\wrtt D_{\R[\seq{t}]|\KK} \subseteq \VV\wrtt D_{\R[\seq{t}]|\KK} \subseteq D_{\R[\seq{t}]|\KK}[w]$. We have that
         \begin{align*}
	 \gr_{\mathcal{G}'}(\VV^+D_{\R[\seq{t}]|\KK}) & \cong  \KK[\seq{x},\seq{y}, t_1^{a_1} \cdots t_{\ell}^{a_\ell} u_1^{b_1} \cdots u_{\ell}^{b_\ell}  w^j]_{\{|\pt{a}| \geq |\pt{b}| + \max\{j,0\}\} }\\
	 &\cong \KK[\seq{x},\seq{y}, t_1w,\ldots, t_\ell w, w^{-1}, t_i u_j]_{\{ i,j\} },
	 \end{align*}
	 and by the same argument above, we are done.
	\end{proof}

Using the algebras introduced in  \Cref{Rees}, we can give another characterization of the stabilization conditions appearing in the definition of the $V$-filtration. 

\begin{lemma}\label{LemmaEquivVfilt}
Suppose that $\KK$ has  characteristic zero. Let $\T$ be a commutative  Noetherian   $\KK$-algebra,  and let $\T[\seq{t}]=\T[\seq[\ell]{t}]$ be a polynomial ring over $\T$.
	Given a $D_{\T[\seq{t}]|\KK}$-module  $M$, a decreasing filtration $\{V^{\alpha}\wrtt M\}_\alpha$ on $M$ indexed by $\alpha\in \QQ$
	satisfies conditions \textup{(1)} and \textup{(5)} of \Cref{VfilM} if
	\begin{enumerate}[label={\textup{(\alph*)}}]
	\item $\VV^\alpha\wrtt M$ is a Noetherian $\VV^+\wrtt D_{\T[\seq{t}]|\KK}$-module for all $\alpha\in \QQ\cap [0,1)$.  
	\end{enumerate}
\end{lemma}
     
\begin{proof}
We first show that  (a) implies (1). Let $\{ M_j\}$ an ascending chain of $V^0_{\overline{t}} D_{\T[\seq{t}]|\KK}$-submodules of 
$V^\alpha_{\overline{t}} M$. We consider the induced chain $\{  \VV^+\wrtt D_{\T[\seq{t}]|\KK}  M_j\}$.
By (a), there exists an integer $a$ such that  for every $j>a$, we have $ \VV^+\wrtt D_{\T[\seq{t}]|\KK}  M_j= \VV^+\wrtt D_{\T[\seq{t}]|\KK}  M_a$. By the grading, we obtain that  $M_j=M_a$ as 
$V^0_{\overline{t}} D_{\T[\seq{t}]|\KK}$-modules. We conclude that $V^{\alpha}\wrtt M$ is a Noetherian left $V^{0}\wrtt {D_{\T[\seq{t}]|\KK}}$-submodule of $M$.

We now show that (a) implies (5). Let $\VV^\alpha\wrtt M = \sum_{i=1}^s \VV^+\wrtt D_{\T[\seq{t}]|\KK} m_i w^{d_i}$ as above. For $d>\max\{d_1,\ldots,d_s \}$, using the previous equation, we obtain that
\begin{align*}
V^{\alpha+d+1}\wrtt M & = \sum V^{d-d_i+1}\wrtt D_{\T[\seq{t}]|\KK} m_i \\
& = \sum \ideal{\seq{t}} \act V^{d-d_i}\wrtt D_{\T[\seq{t}]|\KK} m_i = \ideal{\seq{t}} \act V^{\alpha+d}\wrtt M.\qedhere
\end{align*}
\end{proof}

\subsection{$V$-filtrations for differentially extensible summands}

The main result of this section is the existence of $V$-filtrations for differentially extensible summands.

\begin{theorem}\label{ThmExistenceVfil}
Let $\T$ be a commutative Noetherian ring containing a field $\KK$ of characteristic zero. 
Let $\A$ be a $\KK$-subalgebra of $\T$  such that $\A$ is a direct summand of $\T$, and for which the inclusion $\A \subseteq \T$ is differentially extensible. 
Let $M$ be a $D_{\A[\seq{t}]|\KK}$-module that is a differentially extensible summand of a $D_{\T[\seq{t}]|\KK}$-module $N$.
If $V^\alpha\wrtt N$ is a $V$-filtration on $N$ along the ideal $\ideal{\seq{t}}=\ideal{\seq[\ell]{t}}$ over $\T[\seq{t}]$, then $W^\alpha\wrtt M \coloneqq V^\alpha\wrtt N \cap M$ is a $V$-filtration on $M$ along $\ideal{\seq{t}}$ over $\A[\seq{t}]$.
\end{theorem}

\begin{proof}
We show that $W^\alpha\wrtt M$ is a $V$-filtration for $M$ by verifying Conditions~(2), (3), (4), and (6) of \Cref{VfilM}, and Condition~(a) of \Cref{LemmaEquivVfilt}.

The fact that Condition~(2) holds follows from the fact that $V^\alpha N$ is exhaustive.

To see that Condition~(3) holds, observe that since  $V^\alpha\wrtt N $ is a $V$-filtration for $N$, there exists some $m\in \NN$ for which $V^\alpha\wrtt N$ is constant on each interval of the form $(\frac{n}{m}, \frac{n+1}{m}]$, where $n\in \ZZ$.
Consequently, $W^\alpha\wrtt M =V^\alpha\wrtt N \cap M$ is constant on each such interval.

Toward verifying Condition~(4), we first show that $W^\alpha\wrtt M $ is a $V^{0}\wrtt {D_{\A[\seq{t}]|\KK}}$-submodule of $M$.
Let
\[
   \delta = \sum_{|a|\geq |b|} \delta_{a,b} t^a \pd{}^b \in V^{0}\wrtt {D_{\A[\seq{t}]|\KK}},
\]
where $\delta_{a,b} \in D_{\A|\KK}$, $t^a=t_1^{a_1} \cdots t_\ell^{a_\ell}$,  $\pd{}^b = \pd{t_1}^{b_1} \cdots \pd{t_\ell}^{b_\ell}$, and the sum is taken over finitely many $a,b\in \NN^\ell$. 
For each $\delta_{a,b}$, there exists $\widetilde{\delta}_{a,b}\in D_{\T|\KK}$ such that 
$\widetilde{\delta}_{a,b}|_\A=\delta_{a,b}$; set
\[\widetilde{\delta}=\sum_{|a|\geq |b|}\widetilde{\delta}_{a,b} t^a \pd{}^b.
\]
Given $w\in W^\alpha\wrtt M$, we have $\delta \act w = \widetilde{\delta}\act w$, as $M$  is a differentially extensible summand of $N$.
Since $\widetilde{\delta} \in V^{0}\wrtt {D_{\T[\seq{t}]|\KK}}$,
we obtain that 
$\widetilde{\delta} \act w\in V^\alpha\wrtt N$, hence $\delta \act w \in W^\alpha\wrtt M$.

To conclude that Condition~(4) holds,  it suffices to show that $t^a \pd{}^b \act W^\alpha\wrtt M \subseteq  W^{\alpha +|a|-|b|}\wrtt M$, but this follows from the string of containments
\begin{align*}
t^a \pd{}^b \act W^\alpha\wrtt M = t^a \pd{}^b \act ( V^\alpha\wrtt N\cap M) &\subseteq  
(t^a \pd{}^b \act  V^\alpha\wrtt N)\cap M \\
&\subseteq 
(V^{\alpha +|a|-|b|}\wrtt N) \cap M=W^{\alpha +|a|-|b|}\wrtt M.
\end{align*}

From the definition, it is clear that Condition~(6) holds, since 
\[ \frac{W^\alpha\wrtt M}{\bigcup_{\beta > \alpha} W^{\beta}\wrtt M} = \frac{V^\alpha\wrtt N \cap M}{\bigcup_{\beta > \alpha} V^{\beta}\wrtt N \cap M} \hookrightarrow \frac{V^\alpha\wrtt N}{\bigcup_{\beta > \alpha} V^{\beta}\wrtt N},\]
and
$\sum_{i=1}^r \pd{t_i} t_i - \alpha$ acts nilpotently on the module on the right-hand side. 

Finally, to verify that Condition~(a) holds, we show that the following stronger condition holds: $\mathfrak{W}^\alpha\wrtt M $ is a left-Noetherian  $\VV^+ D_{\A[\seq{t}]|\KK}$-module.
 It suffices to show that if $M_1, M_2 \subseteq \mathfrak{W}^\alpha\wrtt M $ are $\VV^{\alpha}\wrtt {D_{\A[\seq{t}]|\KK}}$-modules such that 
\[
(\VV^{+}\wrtt {D_{\T[\seq{t}]|\KK}})M_1=
(\VV^{+}\wrtt{D_{\T[\seq{t}]|\KK}} )M_2,
\]
then $M_1=M_2$,
because $\VV^\alpha\wrtt N$ is  a left-Noetherian  $\VV^{+}\wrtt {D_{\T[\seq{t}]|\KK}}$-module.
We denote by $\Theta\colon N\to M$ a differential splitting.
Let $w\in M_1$.
Since $\VV^{+}\wrtt {D_{\T[\seq{t}]|\KK}} M_1=\VV^{+}\wrtt {D_{\T[\seq{t}]|\KK}} M_2,$ there exists
$\xi_1,\ldots,\xi_s\in D_{\T|\KK}$ and $w_1,\ldots,w_s \in
M_2$ such that $w=\xi_1 \act w_1+\cdots +\xi_s \act w_s$. We apply
$\Theta$ to both sides to obtain
\begin{align*}
w'  &=\Theta(w')
=\Theta\left(\xi_1 \act w_1+\cdots +\xi_s \act w_s\right) \\
&=\Theta(\xi_1 \act w_1)+\cdots + \Theta(\xi_s \act w_s)
=\widetilde{\xi_1} \act w_1+\cdots +\widetilde{\xi_s} \act w_s,
\end{align*}
where the last equality follows from the fact that $\Theta$ is a
differential splitting compatible with $\beta.$ Thus, $w'\in \VV^{+}\wrtt{D_{\A[\seq{t}]|\KK}} M_2=M_2.$ Hence, $M_1\subseteq M_2.$ Similarly, one can show that $M_2\subseteq M_1$.
\qedhere
\end{proof}

\begin{corollary}\label{CorExistVfilEDS}
Suppose that $\KK$ has  characteristic zero. 
Let $\R$ be a polynomial ring over $\KK$, and let $\A$ be a $\KK$-subalgebra that is a direct summand of $\R$, and such that the inclusion $\A \subseteq \R$ is differentially extensible. 
Let $M$ be a $D_{\A[\seq{t}]|\KK}$-module that is a differentially extensible summand of a regular holonomic $D_{\R[\seq{t}]|\KK}$-module $N$ with quasi-unipotent monodromy.
Then there exists a $V$-filtration on $M$ along the ideal $\ideal{\seq{t}}=\ideal{\seq[\ell]{t}}$. 
\qed
\end{corollary}

Finally, we aim to define a $V$-filtration on a $D_{\A |\KK}$-module $M$ along $\seq{f} = \seq[\ell]{f} \in \A$, where $M$ is a differential direct summand of a $D_{\R |\KK}$-module $N$.
The direct image of $N$ under the graph embedding $i_{\seq{f}}$ is the local cohomology module $H^\ell_{\idtminusf} (N[\seq{t}])$, where $\idtminusf = \ideal{t_1-f_1,\ldots,t_\ell-f_\ell}$.
This $D_{\R[\seq{t}]|\KK}$-module has $H^\ell_{\idtminusf}(M[\seq{t}])$ as a differential direct summand by \Cref{PropLcExt}, and we use this fact in defining our desired $V$-filtration.
In \Cref{CorVfilGens}, we show that the $V$-filtration depends on the ideal generated by $\seq{f}$, but not on the particular choice of its generators.

\begin{definition}[$V$-filtration on a $D_{\A|\KK}$-module along $\seq{f}$] \label{DefVfilIdeal}
Suppose that $\KK$ has  characteristic zero. 
Let $\A$ be a commutative Noetherian $\KK$-algebra. 
Given indeterminates $\seq{t} = \seq[\ell]{t}$, and $\seq{f}=\seq[\ell]{f}\in \A$, consider the ideal 
 $\idtminusf$ of the polynomial ring $\A[\seq{t}]$ generated by $t_1-f_1,\ldots, t_\ell-f_\ell$. 
For a $D_{\A|\KK}$-module  $M$, let $M'$ denote the $D_{\A[\seq{t}]|\KK}$-module $H^{\ell}_{\idtminusf}(M[\seq{t}])$, and 
 identify $M$ with the isomorphic module $0 :_{M'} \idtminusf \subseteq M'$. Suppose that $M'$ admits a $V$-filtration along $\ideal{\seq{t}}$ over $\A[\seq{t}]$.
Then the \emph{$V$-filtration on $M$ along $\seq{f}$} is defined, for $\alpha \in \QQ$, as
 \[V^{\alpha}\wrtf M\coloneqq V^{\alpha}\wrtt M' \cap M = ( 0 :_{V^{\alpha}\wrtt M'} \idtminusf ).\]
\end{definition}

Using the proof of the previous theorem, we show that a $V$-filtration over $\A$ along $\seq{f}$ only depends on the ideal $I=\ideal{\seq[\ell]{f}}$ and not on the generators chosen.

\begin{corollary}\label{CorVfilGens}
Suppose that $\KK$ has  characteristic zero. 
Let $\R$ be a polynomial ring over $\KK$, and let $\A$ be a $\KK$-subalgebra that is a direct summand of $\R$, for which the inclusion $\A \subseteq \R$ is differentially extensible. 
Let $M$ be a $D_{\A[\seq{t}]|\KK}$-module that is a differentially extensible summand of a regular holonomic $D_{\R[\seq{t}]|\KK}$-module $N$ with quasi-unipotent monodromy.
Suppose that $\seq{f}=\seq[\ell]{f}$  and  $\seq{g}=\seq[m]{g}$  both generate an ideal $I\subseteq \A$.
Set $M'=H^{\ell}_{\idtminusf}(M[\seq{t}])$ 
  and $M''=H^{m}_{\ideal{\seq{u}-\seq{g}}}(M[\seq{u}])$, for another sequence of indeterminates $\seq{u}=\seq[m]{u}$.
We identify $M$ with both $0 :_{M'} \ideal{\seq{t}-\seq{f}}$ and  $0 :_{M''} \ideal{\seq{u}-\seq{g}}$.
Then $V^\alpha_{\ideal{\seq{t}}} M'\cap M=V^\alpha_{\ideal{\seq{u}}} M''\cap M$.
In particular, the $V$-filtration of $M$ along $\seq{f}$ depends only on the ideal $I$.
\end{corollary}

\begin{proof}
Let  $N'=H^{\ell}_{\ideal{\seq{t}-\seq{f}}}(N[\seq{t}])$   and $N''=H^{m}_{\ideal{\seq{u}-\seq{g}}}(N[\seq{u}])$. 
We observe that $M'$ is a differentially extensible summand of $N'$, and  $M''$ is a differentially extensible summand of $N''$ by \Cref{PropLcExt}.
Since both $N'$ and $N''$ have $V$-filtrations (along $\seq{t}$ and $\seq{u}$), so do $M'$ and $M''$ by \Cref{ThmExistenceVfil}.
We have that  $N$ is identified with  $0 :_{N'} \ideal{\seq{t}-\seq{f}}$ and  $0 :_{N''} \ideal{\seq{u}-\seq{g}}$.
Then $V^\alpha_{\ideal{\seq{t}}} N'\cap N=V^\alpha_{\ideal{\seq{u}}} N''\cap N$  \cite[Proposition~2.6]{BMS2006a}, and thus we have the equalities
\begin{align*}
V^\alpha_{\ideal{\seq{t}}} M'\cap M &=(V^\alpha_{\ideal{\seq{t}}} N' \cap M') \cap M =V^\alpha_{\ideal{\seq{t}}} N' \cap M' \cap N \cap  M  \\
&=(V^\alpha_{\ideal{\seq{t}}} N'  \cap N) \cap M' \cap M =(V^\alpha_{\ideal{\seq{u}}} N''  \cap N) \cap M'' \cap M \\ 
&=(V^\alpha_{\ideal{\seq{u}}} N'' \cap M'') \cap M  =V^\alpha_{\ideal{\seq{u}}} M''\cap M.\qedhere
\end{align*}
\end{proof}

\subsection{Hodge  ideals} \label{Hodge}

Given  $f\in R=\CC[x_1,\ldots,x_d]$, the $D_{R|\CC}$-module
of regular functions on $X=\Spec(R)$ with poles along the divisor $Y={\rm div}(f)$ is given by 
$\cO_X(*Y)=R_f$. This module  is equipped with the Hodge filtration $\cF_\bullet (\cO_X(*Y))$ \cite{Saito90,MP19}, which is an invariant measuring the singularities of $Y$.  Musta\c{t}\u{a} and Popa \cite{MP19} defined a sequence of ideals, the \emph{Hodge ideals} of $f$, that encode data of the Hodge filtration.  Later on, they extended this notion to $\QQ$-divisors  \cite{MPBirational,MPVfil}. In this framework, they considered a filtration $\cF_\bullet \cO_X (*Y) \cdot f^{1-\lambda}$ on the free $D_{R|\CC}$-module of rank one $ \cO_X (*Y) \cdot f^{1-\lambda}= R_f  \cdot f^{1-\lambda}$, where $\lambda \in \QQ_{\geq 0}$. If we denote the support of $Y$ as $Z=Y_{\rm red}$,  then  the sequence of \emph{Hodge ideals} of $f^\lambda$, which we denote as $I_k(f^\lambda)$, $k\geq 0$, 
is defined by the equations:
$$
\cF_k \cO_X (*Y) \cdot f^{1-\lambda} =I_k(\lambda Y) \otimes \cO_X (k Z + Y) \cdot f^{1-\lambda}.
$$
We point out that Hodge ideals can be also defined for ideals of $R$, and $\lambda \in (0,1] \cap \QQ$  \cite{MPIdeals}. This description boils down to the computation of Hodge ideals of $\QQ$-divisors.

An interesting feature is that Hodge ideals of $\QQ$-divisors are characterized completely in terms of $V$-filtrations. For simplicity we will just consider the case where the divisor is reduced. Set $Q_j(x):= x(x+1) \cdots (x+j-1) \in \ZZ[x]$. 


\begin{theorem}[{\cite{MPVfil}}]\label{ThmDefHodgeVfil}
Let $Y$ be a reduced divisor defined by $f\in R$. Then for all $\lambda \in \QQ_{\geq 0}$ and $k\geq 0$, we have
{ $$
I_k(f^\lambda)=\left\{  g=\sum^k_{j=0} Q_j(\lambda) f^{k-j} h_j \in R \ \Bigg|\ {\large \substack{ \exists h_0,\dots ,h_k\in R \ \mathrm{ s.t. } \\ \sum^k_{j=0} h_j \partial_t^j \frac{1}{f-t}\in V^\lambda_{\ideal{t}} H^1_{\ideal{f-t}}(R[t])}}  \right\}.
$$ }
\end{theorem}

We refer to the work of Musta\c{t}\u{a} and Popa \cite{MPProperties,MP19,MPBirational,MPVfil}  for properties of these ideals.

Theorem~\ref{ThmDefHodgeVfil} allows us to define Hodge ideals whenever $V$-filtrations are defined.

\begin{definition}
Suppose that $\KK$ has  characteristic zero. 
Let $T$ be a $\KK$-algebra  that admits   a $V$-filtration along a reduced $f\in T$.
For any $\lambda \in \QQ_{\geq 0}$, and every  $k\geq 0$,  the  $k$-th Hodge ideal of $f^\lambda$ is defined by
$$
I^T_k(f^\lambda)=\left\{  g=\sum^k_{j=0} Q_j(\lambda) f^{k-j} h_j \in T \ \Bigg|\ {\large \substack{\exists h_0,\dots,h_k\in T\ \mathrm{ s.t.} \\ \sum^k_{j=0} h_j \partial^j_t \frac{1}{f-t}\in V^\lambda_{\ideal{t}} H^1_{\ideal{f-t}}(T[t])}}  \right\}.
$$ 
If the ring $T$ is clear from the context, we write $I_k(f^\lambda)$ instead of $I^T_k(f^\lambda)$.
\end{definition}

Observe that the conditions in the definition of the $V$-filtration imply that for fixed $f$ and $k$, the collection of ideals $\{I^T_k(f^\lambda)\}$ is decreasing in $\lambda$, and that the set of $\lambda$ for which $I^T_{k}(f^\lambda)\neq I^T_{k}(f^{\lambda'})$ for all $\lambda'>\lambda$ is a discrete set of rational numbers.
From our results on $V$-filtrations, we can show that Hodge ideals exist for 
extensible direct summands of polynomial rings.

\begin{corollary}\label{CorHodge}
Suppose that $\KK$ has  characteristic zero. 
Let $\R$ be a polynomial ring over $\KK$, and let $\A$ be a $\KK$-subalgebra that is a direct summand of $\R$, for which the inclusion $\A \subseteq \R$ is differentially extensible. 
Then $I_k^A(f^\lambda)$ exists for all $f\in A$ and  $\lambda\in \QQ_{\geq 0}$.
Furthermore,
\begin{equation}\label{EqRestrictionHI}
I^A_k(f^\lambda)=I^R_k(f^\lambda)\, \cap \, A. 
\end{equation}
\end{corollary}
\begin{proof}
The existence of Hodge ideals follows from Corollary \ref{CorExistVfilEDS}.
We prove \eqref{EqRestrictionHI} by showing both containments. First take $g\in I^A_k(f^\lambda)$.  There exists $ h_0,\dots,h_k\in A\subseteq R$ such that  $g=\sum^k_{j=0} Q_j(\lambda) f^{k-j} h_j$ and
$$\sum^k_{j=0} h_j \partial^j_t \frac{1}{f-t}\in V^\lambda_{\ideal{t}} H^1_{\ideal{f-t}}(A[t])\subseteq V^\lambda_{\ideal{t}} H^1_{\ideal{f-t}}(R[t])$$
so that $ g\in I^R_k(f^\lambda)$.

On the other hand, for $g\in I^R_k(f^\lambda)\cap A$,
there exist $ h_0,\dots,h_k\in R$ such that $g=\sum^k_{j=0} Q_j(\lambda) f^{k-j} h_j$ and
$$\sum^k_{j=0} h_j \partial^j_t \frac{1}{f-t}\in V^\lambda_{\ideal{t}} H^1_{\ideal{f-t}}(A[t])\subseteq V^\lambda_{\ideal{t}} H^1_{\ideal{f-t}}(R[t]).$$

Let $\beta:R\to A$ a splitting, and $\theta:H^1_{\ideal{f-t}}(R[t])\to H^1_{\ideal{f-t}}(A[t])$ a differentiable splitting.
Then
\begin{align*}
\theta\left(\sum^k_{j=0} h_j \partial_t^j \frac{1}{f-t}\right)&=
\sum^k_{j=0} \theta\left(h_j \partial_t^j \frac{1}{f-t}\right)\\
&=\sum^k_{j=0}( \beta\circ h_j \partial_t^j) \frac{1}{f-t}\\
&=\sum^k_{j=0}( \beta h_j) \partial_t^j \frac{1}{f-t}.
\end{align*}
Since  $\beta( h_0),\dots,\beta(h_k)\in A$, we conclude that $g\in I^A_k(f^\lambda)$.
\end{proof}

We now recall a definition related to the log canonical threshold.

\begin{definition}
Suppose that $\KK$ has  characteristic zero. 
Let $T$ be a $\KK$-algebra such that $b^T_f(s)$ exists for every $f\in T$.
The \emph{minimal exponent} of $f$, $\widetilde{\alpha}_f$, is the negative of the largest root of $b_f(s)/(s+1)$.
\end{definition}

As a consequence of \Cref{CorHodge}, we can relate Hodge ideals and minimal exponent. 

\begin{theorem}\label{CorMinExp}
Suppose that $\KK$ has  characteristic zero. 
Let $\R$ be a polynomial ring over $\KK$, and let $\A$ be a $\KK$-subalgebra that is a direct summand of $\R$, for which the inclusion $\A \subseteq \R$ is differentially extensible. 
Given $\lambda\in(0,1]\cap \QQ$,  $I^A_k(f^\lambda)=0$ if and only if $k\leq \widetilde{\alpha}_f-\lambda$.
\end{theorem}
\begin{proof}
Since this equivalence is known for polynomial rings \cite[Theorem C]{MPVfil},
the result  follows from Theorems~\ref{ThmBSEqual} and Corollary \ref{CorHodge}.
\end{proof}


\section{Multiplier  ideals} \label{Multiplier}

Given a field $\KK$ of characteristic zero, let $\R$ be a polynomial ring over $\KK$ and $X=\Spec(\R)$. To any ideal $I\subseteq \R$, we associate a family of \emph{multiplier ideals} $\cJ_\R(I^\lambda)$ parameterized by nonnegative real numbers $\lambda$.
There exists a discrete sequence of rational numbers
\[
 0 = \lambda_0 < \lambda_1 < \lambda_2 < \cdots 
\]
called the \emph{jumping numbers} of $I$, 
for which $\cJ_\R(I^\lambda)$ is constant for $\lambda_i \leq \lambda < \lambda_{i+1}$, and $\cJ_\R(I^{\lambda_i})\varsupsetneq\cJ_\R(I^{\lambda_{i+1}})$. 
The first nonzero jumping number, $\lambda_1$, is known as the \emph{log canonical threshold of $I$}, denoted $\lct_\R(I)$. 

We begin this section by recalling the definition of the multiplier ideals in this setting, 
and we refer to \cite{Laz2004} for details, and any unexplained terminology.	
	
\begin{definition}[Multiplier ideals over a polynomial ring]\label{multiplier}
Let $\R$ be a polynomial ring over a field $\KK$ of characteristic zero, let $I$ be an ideal of $\R$, and let $X = \Spec(\R)$.
Let $\pi\colon Y\to X$ be a log resolution of $I$ such that $I\cO_Y=\cO_Y(-F_\pi)$, where 
$F_\pi=\sum r_i E_i$, and $K_{Y/X}=\sum b_i E_i$ is the {relative canonical divisor}.
The \emph{multiplier ideal} of $I$ with exponent $\lambda\in \RR_{\geq 0}$ is defined by 
\[
\cJ_\R(I^\lambda)=\pi_* \cO_Y(\lceil K_{Y/X}- \lambda F_\pi\rceil)=\{h\in \R : \ord_{E_i} (\pi^* h) \geq \lfloor \lambda r_i  - b_i\rfloor\; \forall i  \}.
\]
\end{definition}
	
In our setting, there is a characterization of multiplier ideals in terms of the $V$-filtration that, roughly speaking, says that the chain of multiplier ideals is
essentially the $V$-filtration of $\R$ along the ideal $I$  \cite{BudurSaito05, BMS2006a}. 
In fact, this relation has been used to develop algorithms to compute multiplier ideals, based on existent algorithms for computing the Bernstein--Sato polynomial that use Gr\"obner bases over the ring of differential operators \cite{BL-Alg,Shibuta}. 		
	
\begin{theorem}[{\cite[Theorem~1, Corollary~2]{BMS2006a}}]\label{ThmMultIdealVersions}
Let $\R$ be a polynomial ring over a field $\KK$ of characteristic zero, and let $I$ be an ideal of $\R$. 	
	Given a real number $\lambda\geq 0$, the following ideals coincide\textup: 
\begin{enumerate}
\item $\cJ_\R(I^\lambda)$,
\item $\bigcup_{\alpha>\lambda}V^\alpha \R$, where the $V$-filtration is taken along  $I$, and 
\item  $\{ g\in \R : \gamma>\lambda \hbox{  if } b_{I,g}^\R(-\gamma)=0\}$.
\end{enumerate}
\end{theorem}

The definition of the multiplier ideals can be easily extended to the case that $X$ is $\QQ$-Gorenstein, where it is possible to construct the relative canonical divisor.
A version of multiplier ideals in normal varieties was defined  by de Fernex and Hacon \cite{dFH}, and  later extended by Chiecchio, Enescu, Miller, and Schwede, to include a boundary divisor term  \cite{CEMS}.

Our aim in this section is to study multiplier ideals for direct summands of coordinate rings of smooth varieties, and to extend \Cref{ThmMultIdealVersions} to this setting. Our most general result in this direction is the following.

\begin{theorem}\label{ThmMultV}
Let $\R$ be a polynomial ring over a field $\KK$ of characteristic zero, and let $\A\subseteq \R$ be a differentially extensible inclusion of finitely generated $\KK$-algebras, such that $\A$ is a direct summand of $\R$.	
Then for every ideal $I$ of $\A$, and every real number $\lambda\geq 0$,  the following ideals coincide\textup:
\begin{enumerate}
\item $\cJ_\R((I\R)^\lambda)\cap \A$, 
\item $\bigcup_{\alpha>\lambda}V^\alpha \A$,  where the $V$-filtration is taken along $I$, and 
\item  $\{ g\in \A : \gamma>\lambda \hbox{  if } b^\A_{I,g}(-\gamma)=0\}$.
\end{enumerate}
\end{theorem}

\begin{proof}
From the equality $ b^\A_{I,g}(s)= b^\R_{I\R,g}(s)$ given in \Cref{ThmBSEqual}, we deduce 
\[
\{ g\in \A : \gamma>\lambda \hbox{  if } b^\A_{I,g}(-\gamma)=0\}=
\{ g\in \R : \gamma>\lambda \hbox{  if } b^\R_{I\R,g}(-\gamma)=0\}\cap \A.
\]
Moreover, $\bigcup_{\alpha>\lambda}V^\alpha \A$ is the intersection  $(\bigcup_{\alpha>\lambda}V^\alpha \R)\cap \A$, where the $V$-filtrations are taken along $I$ in $\A$, and along $I\R$ in $\R$. 
The claim then follows from \Cref{ThmMultIdealVersions}.
\end{proof}


Of course, it would be desirable if $\cJ_\R((I\R)^\lambda)\cap \A$ were equal to $\cJ_\A(I^\lambda)$, and we show this, under certain conditions, in \Cref{theorem-multiplier}. To prove this result, we make a detour through positive characteristic methods, relying on the well-known relation between multiplier ideals and test ideals \cite{Karen2000,Hara2001,HY2003}.

\subsection{Test ideals and Cartier extensibility}

Throughout this subsection, $\T$ denotes a  commutative Noetherian ring of prime characteristic $p$. 
The $e$-th iteration of the Frobenius endomorphism $F^e$ on $\T$, given by $r\mapsto r^{p^e}$, yields a new $\T$-module structure on $\T$ by restriction of scalars.
We denote this module action on $\T$ by $F^e_* \T$, so that for $r\in \T$ and $F^e_*x \in F^e_* \T$, 
\[r \cdot F^e_*x= F^e_*(r^{p^e}x) \in F^e_* \T.\]
We say that $\T$ is \emph{$F$-finite} if $F^e_* \T$ is a finitely generated $\T$-module for some (equivalently, all) $e>0$.

The iterated Frobenius endomorphism can be used to classify singularities. 
 
\begin{definition}
   Let $\T$ be an $F$-finite Noetherian domain of prime characteristic.
   \begin{enumerate}
      \item We call $\T$ \emph{$F$-pure} if the inclusion map $\T\to F^e_* \T$ splits for some (equivalently, all) $e>0$.
      \item We call $\T$ \emph{strongly $F$-regular} if for every $r\in \T$, there exists $e\in\NN$ such that the $\T$-linear map $\T\to F^e_* \T$ sending $1$ to $F^e_* r$ splits.
   \end{enumerate}
\end{definition}
		
A regular ring is strongly $F$-regular, and a direct summand of a strongly $F$-regular ring is itself strongly $F$-regular \cite[Theorem~3.1]{HoHuStrong}. 
		
\begin{definition} \label{pelinear: D}
   Let $\T$ be an $F$-finite Noetherian domain of prime characteristic~$p$.
   An additive map $\psi\colon \T\to \T$ is a \emph{$p^{-e}$-linear map} if $\psi(r^{p^e} f)=r\psi(f)$.
   Such a map is also referred to as a \emph{Cartier operator}.
   The set of all the $p^{-e}$-linear maps on $\T$ is denoted $\cC^e_\T$; in the notation of \Cref{pelinear: D}, $\cC^e_\T \cong \Hom_\T(F^e_* \T,\T)$.
\end{definition}

Test ideals first appeared in the theory of tight closure developed by Hochster and Huneke \cite{HH1990, HoHu2}.
Hara and Yoshida subsequently defined test ideals associated to pairs $(\T, I^\lambda)$, where $I$ is an ideal of a regular ring $\T$, and $\lambda$ is a nonnegative real parameter  \cite{HY2003}.
A tight closure-free version of test ideals was developed by Blickle, Musta\c{t}\u{a}, and Smith, using Cartier operators over regular rings \cite{BMS2008, BMS2009}.
This new approach was extended to nonregular rings  \cite{TestQGor, BlickleP-1maps, BB-CartierMod}. 
For strongly $F$-regular rings, we have the following description (see \cite{TT}).

\begin{definition}
   Let $\T$ be an $F$-finite Noetherian domain of prime characteristic~$p$ that is strongly $F$-regular.
   The \emph{test ideal} of an ideal $I$ of $\T$ with respect to a real parameter $\lambda\geq 0$ is defined by
   \[
      \ti{\T}\big(I^{\lambda}\big)=\bigcup_{e>0} \Bigg(\sum_{\psi\in \cC^e_\T} \psi\Big( I^{\lceil p^e \lambda \rceil}\Big)\Bigg).
   \]
\end{definition}
		
Cartier operators are closely related to differential operators, so it is natural to consider our condition of differentiable extensibility in this context, as well.
		
\begin{definition}[Cartier extensibility] \label{Def_Cartier_ext}
   Consider a ring homomorphism $\varphi\colon \A\to \T$,  and fix a Cartier operator $\psi\in \cC^e_{\A}$.  
We say that 
	\begin{enumerate}
		\item  A Cartier operator $\widetilde{\psi}\in \cC^e_{\T}$ \emph{extends} $\psi$ if $\varphi \circ \psi = \widetilde{\psi}\circ \varphi$; i.e., the following diagram commutes:
\[
\xymatrix{
\A\ar[d]_{\psi} \ar[r]^{\varphi} & \T\ar[d]^{\widetilde{\psi}}\\
\A \ar[r]^{\varphi}  & \T }
\]
		If $\varphi$ is simply an inclusion, 
		then this is equivalent to the condition that $\widetilde{\psi}|_\A=\psi$. 

		\item The map $\varphi$ is \emph{Cartier extensible} if for every $\psi\in \cC^e_{\A}$,  there exists  $\widetilde{\psi}\in \cC^e_{\T}$ that extends $\psi$.
		
    \end{enumerate}  
\end{definition}		
	
\begin{lemma}\label{LDE->CE}
   An inclusion $\A\subseteq \T$ of $F$-finite $F$-pure domains of prime characteristic $p>0$
   that is differentially extensible with respect to the level filtration is also Cartier extensible.
\end{lemma}

\begin{proof}
   Let $i_\T\colon \T\to F^e_* \T$ be the inclusion map, and $\pi_\T\colon F^e_* \T\to \T$ a $\T$-linear splitting.
   Likewise, let $i_\A\colon \A\to F^e_* \A$ be the inclusion map, and $\pi_\A$ an $\A$-splitting.
   Given an $\A$-linear map $\psi\colon F^e_* \A \to \A$, we obtain an $\A$-linear map $i_\A \circ \psi$ on $F^e_* \A$.
   Note that $\psi = \pi_\A \circ (i_\A \circ \psi)$.
   By assumption, there is a $\T$-linear map $\beta$ on $F^e_*\T$ such that $\beta|_{F^e_* \A} = i_\A \circ \psi$.
Then since  $\mathrm{im}(\beta|_{F^e_* \A}) \subseteq \A$ and $(\pi_\T)|_{\A}$ is the identity map,
$\pi_\T \circ \beta \in \cC^e_\T$ extends $\psi$.
\end{proof}

\begin{proposition}\label{PropTestIdealRetraction}
   Let $\R$ be a regular $F$-finite domain of prime characteristic $p>0$. Let $\A\subseteq \R$ be an  extension of Noetherian rings, such that $\A$ is a Cartier extensible direct summand of $\R$.
   Then for every ideal $I$ of $\A$, and every real number $\lambda \geq 0$. 
   \[\ti{\A}(I^\lambda)=\ti{\R}((I\R)^\lambda)\cap \A. \]
\end{proposition} 

\begin{proof}
Since $\A$ is a direct summand of $\R$, $\ti{\R}((I\R)^\lambda)\cap \A\subseteq \ti{\A}(I^\lambda)$
\cite[Proposition~4.9]{AMHNB}, and we focus on the other containment. For any $f\in \ti{\A}(I^\lambda)$,  there exists $e\in \NN$ for which
$f\in \cC^e_\A I^{\lceil p^e \lambda \rceil}$. Then there exist $g_1,\ldots, g_\ell\in  I^{\lceil p^e \lambda \rceil}$ and $\psi_1,\ldots, \psi_\ell\in \cC^e_\A$ such that $f=\psi_1 g_1+\cdots + \psi_\ell g_\ell$. Given extensions $\widetilde{\psi}_i\in \cC^e_\R$ of $\psi_i$, 
\[
f=\psi_1 g_1+\cdots + \psi_\ell g_\ell =\widetilde{\psi}_1 g_1+\cdots + \widetilde{\psi}_\ell g_\ell\in\cC^e_\R  I^{\lceil p^e \lambda \rceil}\subseteq \ti{\R}((I\R)^\lambda)
\]
so that $\ti{\A}(I^\lambda)\subseteq \ti{\R}((I\R)^\lambda)\cap \A$.
\end{proof}

\subsection{Multiplier ideals and reduction modulo $p$}
                
The relationship between multiplier ideals and test ideals using reduction to prime characteristic $p$ was originally observed in work of Smith \cite{Karen1997b,Karen2000} and Hara \cite{Hara2001}.  
For the case of pairs, we refer to work of Hara and Yoshida \cite{HY2003}, to that of Takagi \cite{Takagi04} for $\QQ$-Gorenstein varieties, and to that of de Fernex, Docampo, Takagi, and Tucker for the extension to the numerically $\QQ$-Gorenstein case \cite{dFDTT2015}. Herein, we call upon a result of Chiechio, Enescu, Miller, and Schwede that states that multiplier ideals reduce to test ideals via reduction to prime characteristic for a large class of rings \cite{CEMS}.

Before explaining the result in \emph{loc.\,cit.}, we recall the basics on reduction to prime characteristic. For more details, we 
 refer to Hochster and Huneke's notes on tight closure in characteristic zero \cite{HHCharZero}.

\begin{definition}[Model for an ideal in a finitely generated $\KK$-algebra] \label{DefModel}
Suppose that $\A$ is a finitely generated $\KK$-algebra, where $\KK$ has characteristic zero, and fix an ideal $I$ of $\A$. 
 Fix a finitely generated $\ZZ$-subalgebra $B$ of $\KK$,  
  and a $B$-subalgebra $\A_B$ of $\A$ essentially of finite type over $B$,  that satisfy the following conditions:
 $\A_B$ is flat over $B$, $\A_B \otimes_B \KK \cong \A$, and $I_B \A = I$, where $I_B = I \cap \A_B \subseteq \A_B$.
For every closed point $s\in\Spec(B)$, 
we set $\A_s = \A_B \otimes_B \kappa(s)$ and $I_s = I_B \A_s \subseteq \A_s$, where $\kappa(s)$ denotes the residue field of $s$.
We say that $(B,\A_B,I_B)$ is a \emph{model} for $(\KK,\A,I)$.

If, additionally, $\R$ is another finitely generated $\KK$-algebra, and $\varphi\colon \A \to \R$ is a $\KK$-algebra homomorphism, then we say that $(B,\A_B,I_B,\R_B,\varphi_B)$ is a \emph{model} for $(\KK,\A,I,\R,\varphi)$ if $(B,\A_B,I_B)$ is a model for $(\KK,\A,I)$, $(B,\R_B,0)$ is a model for $(\KK,\R,0)$, and $\varphi_B\colon \A_B \to \R_B$ satisfies $\varphi=\varphi_B \otimes_B \KK$. We write $\varphi_s = \varphi_B \otimes_B \kappa(s)$ for $s\in\Spec(B)$ closed.
\end{definition}

\begin{remark}
   Note that if $\A$ is a subalgebra of a polynomial ring $\R$ over a field $\KK$ of characteristic zero, and $\A$ is generated by polynomials with integer coefficients, then one can take $B=\ZZ$ in a model for $\A$, and $\A_{\ZZ}$ the $\ZZ$-algebra with the same generators.
   In this case, a condition on $\A_s$, for all $s$ in a dense open subset of $\Spec(\ZZ)$ is equivalent to this condition on $\A_{\ZZ}/p \A_{\ZZ}$, for all $p \gg 0$. We write $\A_p$ to denote $\A_{\ZZ}/p \A_{\ZZ}$ in this case.
\end{remark}

\begin{definition}[Anticanonical cover]
 Let $\A$ be a normal ring that is either complete and local, or $\NN$-graded and finitely generated over its zeroth graded component. The \emph{anticanonical cover} of $\A$ is the symbolic Rees algebra
 \[\cR (-\omega_\A)= \bigoplus_{n\geq 0} \omega_\A^{(-n)},\] where $\omega_\A$ is the canonical module of $\A$. 
\end{definition}

By  recent results in the minimal model program  \cite[Corollary~1.1.9]{BCHM}, we know that for varieties with KLT singularities, $\cR (-\omega_\A)$ is finitely generated.

\begin{theorem}[{\cite[Theorem~D]{CEMS}}]\label{ThmCEMS}
   Suppose that $\A$ is a finitely generated $\KK$-algebra over an algebraically closed field of characteristic zero such that the anticanonical cover $\cR (-\omega_\A)$ is finitely generated.
Let $I$ be an ideal of $\A$, and fix $\lambda\in \RR_{\ge 0}$. 
Let $(B,\A_B,I_B)$ be a model for $(\KK,\A,I).$
Then
\[
[\cJ_\A (I^\lambda)]_s=\ti{\A_s}(I^\lambda_s)
\]
for every closed point $s$ in a dense open subset of $\Spec(B)$.
\end{theorem}	
 	

Consider the following setup.

\begin{setup}\label{setup: multiplier}
Let $\R$ be a polynomial ri
ng with coefficients in  an algebraically closed field $\KK$ of characteristic zero.
   Suppose that $\A \subseteq \R$ is differentially extensible, that the inclusion map $\varphi$ makes $\A$ a direct summand of $\R$, and that $\A$ is finitely generated over $\KK$ with KLT singularities.
Moreover, let $I$ be an ideal of $\A$, and let $(B,\A_B,I_B,\R_B,\varphi_B)$ be a model for $(\KK,\A,I,\R,\varphi)$. Suppose that for all $s$ in a dense open subset $U$ of $\Spec(B)$, $\A_s$ is a direct summand of $\R_s$ and the inclusion $\varphi_s$ of $\A_s$ into $\R_s$ is also Cartier extensible. 
\end{setup}

In this setting, we fully understand the relationship between the multiplier ideals $\cJ_\A (I^\lambda)$ and $\cJ_\R((I\R)$.

\begin{theorem}\label{theorem-multiplier}
 In the context of \Cref{setup: multiplier}, for every real number $\lambda \geq 0$, 
 \[
\cJ_\A (I^\lambda)=\cJ_\R((I\R)^\lambda)\cap \A.
\]
\end{theorem}

\begin{proof}
   Take $f\in \cJ_\A (I^\lambda)$, and for $s\in \Spec(B)$, let $f_s$ denote the image of $f$ in $\A_s$.
   Then $f_s\in [ \cJ_\A (I^\lambda)]_s$ for every $s$, so $f_s\in \ti{\A_s}(I^\lambda_s)$ for $s \in U$ by \Cref{ThmCEMS}.
Thus, $f_s\in  \ti{\R_s}((I\R_s)^\lambda)\cap \A_s$ by \Cref{PropTestIdealRetraction}. In particular, $f_s\in  \ti{\R_s}((I\R_s)^\lambda)$, so that  $f_s\in  [ \cJ_\R ((I\R)^\lambda)]_s$ for $s \in U$, and $f\in \cJ_ \R((I\R)^\lambda)\cap \A$.

Now take $f\in  \cJ_\R((I\R)^\lambda)\cap \A$, so that $f_s\in [ \cJ_\R (I^\lambda)]_s \cap \A_s$ for every $s$.
Then $f_s\in  \ti{\R_s}((I\R_s)^\lambda)\cap \A_s=\ti{\A_s}(I^\lambda_s)$ for all $s$ in the dense open set described in \Cref{ThmCEMS} and \Cref{PropTestIdealRetraction}.
Then $f_s\in [ \cJ_\A (I^\lambda)]_s$  for such $s$, and so $f\in \cJ_\A (I^\lambda)$.
\end{proof}

The hypotheses that the inclusion  $\A\subseteq \R$ is differentially extensible, and split as $\A$-modules, are essential for \Cref{theorem-multiplier}, as the following example shows.

\begin{example} \label{ex_det}
 Given  $\A=\CC[x^2,y^2] \subseteq \R= \CC[x,y]$, and $I=\ideal{x^2,y^2}\subseteq \A$, then $\cJ_\A(I)=\A$, whereas $\cJ_\R(I\R)=\ideal{x,y}$. In particular, $\cJ_\R(I\R)\cap \A \neq \cJ_\A(I)$.
\end{example}

The following example indicates that the assumption in \Cref{theorem-multiplier} that, after reduction to prime characteristic, $\A_p\subseteq \R_p$ is a direct summand that is also Cartier extensible for $p\gg 0$, might not be necessary. 

\begin{example}\label{ExDetl3x3}  Let $X,Y,Z$ be  $3\times 3$, $3 \times 2$, and $2\times 3$ matrices of indeterminates, respectively. Let $\T=\CC[X]$, $\A=\T/\ideal{\det(X)}$, and $\R=\CC[Y,Z]$. Then the $\CC$-algebra map sending $x_{ij}$ in $\A$ to the $(i,j)$-entry of the matrix product $YZ$ realizes $\A$ as an extensible direct summand of $\R$ \cite[IV~1.9]{LS}.
   Let $I\subseteq \A$ be the ideal generated by the $2\times 2$ minors~of~$X$. 

Using the methods of Johnson  \cite[Chapters~4~and~5]{Johnson} along with techniques of Takagi \cite[Theorem~3.1]{TakagiAdj}, one can compute $\cJ_\A(I^\lambda)=\m^{\lfloor 2 \lambda\rfloor-5} \cap I^{\lfloor \lambda\rfloor-1}$, where $\m$ is the ideal generated by the entries of $X$. We have that $\cJ_\R(I\R^{\lambda})=(I_1(Y)I_1(Z))^{\lfloor 2 \lambda \rfloor -5} \cap (I_2(Y)I_2(Z))^{\lfloor \lambda \rfloor -1}$ using \cite[Theorem~5.4]{Johnson}  and \cite[Variant~2.5]{Subadd}, where $I_i(-)$ denotes the ideal of $i\times i$-minors of a matrix. It is then easy to verify that $\cJ_\R(I\R^{\lambda})\cap \A = \cJ_\A(I^{\lambda})$.

However, for $p$ prime, the inclusion $\A_p \subseteq \R_p$ is not a direct summand. Indeed, suppose it were. Then there would be an injection $H^8_{\m_p}(\A_p)\hookrightarrow H^8_{\m_p \R_p}(\R_p)$.  However, $H^8_{\m_p}(\A_p)\neq 0$ by Grothendieck nonvanishing, and since $\mathrm{pd}_{\R_p}(\m_p \R_p) = 7$, we know that $H^8_{\m_p \R_p}(\R_p)=0$ by Peskine--Szpiro vanishing, yielding a contradiction.
\end{example}

Motivated by \Cref{ThmMultV,theorem-multiplier} and \Cref{ExDetl3x3}, we wonder whether is it possible to drop the condition of Cartier extensibility after reduction to prime characteristic. Namely, we pose the following question:

\begin{question}\label{QuestionConjecture}
Let $\A$ be a differentially extensible  summand of a polynomial ring over a field of characteristic zero.  Given an ideal $I$ of $\A$, and a real number $\lambda \geq 0$, 
do the following ideals coincide?
    \begin{enumerate}
        \item $\cJ_\A(I^\lambda)$,
        \item $\bigcup_{\alpha>\lambda}V^\alpha \A$, where the $V$-filtration is taken along  $I$, and 
        \item  $\{ g\in \A : \gamma>\lambda \hbox{  if } 
        b^\A_{I,g}(-\gamma)=0\}$.
    \end{enumerate}
\end{question}

The following consequence of \Cref{theorem-multiplier} may be known to experts in some cases. 

\begin{corollary}
   In the context of \Cref{setup: multiplier} we fix a real number $\lambda < 1$ and an ideal $I=\ideal{g_1,\ldots,g_\ell}$ of $\A$. 
   Moreover, let $g$ be a general $\KK$-linear combination of $g_1, \ldots, g_\ell$.  Then
$
\cJ_\A (I^\lambda)= 
\cJ_\A (f^\lambda)$, and as a consequence,
$\lct(f)=\min\{\lct(I),1\}$.
\end{corollary}
\begin{proof}
Since
$\cJ_\R((I\R)^\lambda)=\cJ_\R(f^\lambda)$ \cite[Theorem 9.2.28]{Laz2004}, then by \Cref{theorem-multiplier}, 
\[
\cJ_\A (I^\lambda)=\cJ_\R((I\R)^\lambda)\cap \A=\cJ_\R(f^\lambda)\cap \A=\cJ_\A (f^\lambda). \qedhere
\]
\end{proof}

The following is an extension to our setting of results for smooth varieties that relate the jumping numbers of multiplier ideals with the roots of Bernstein--Sato polynomials \cite{ELSV2004,BMS2006a}. 

\begin{theorem}\label{theorem-jn-roots}
Under \Cref{setup: multiplier}, the log canonical threshold of
 $I$ in $\A$ is the smallest root $\alpha_I$ of
 $b_I(-s)$, and every jumping number of $I$ in $[\alpha_I,\alpha_I +1)$ is a root of $b_I(-s)$.
\end{theorem}

\begin{proof}
Since $\cJ_\A (I^\lambda)=\cJ_\R((I\R)^\lambda)\cap \A$, we have that $\cJ_\A (I^\lambda)=\A $ if and only if $\cJ_\R((I\R)^\lambda)=\R$. As a consequence, 
$\lct_\A(I)=\lct_\R(I\R)$. Since $b^\A_I(-s)=b^\R_{I\R}(-s)$ by \Cref{ThmBSEqual}, the first claim follows.
Furthermore, every jumping number of $I$ in $\A$ is also a jumping number of $I\R$ in $\R$.  Then the second claim follows again from \Cref{ThmBSEqual}.
\end{proof}

The above results apply to pointed normal affine toric varieties.
Recall that by \Cref{toricODE}, \Cref{toricLDE}, and 
\Cref{toricODELDE}, there is an embedding of $\A$ into a polynomial ring $\R$ over $\KK$ such that, $\A\subseteq \R$ is order-differentially extensible, 
and $\A_p \subseteq \R_p$ is level-differentially extensible for $p\gg 0$.
The ring $R$ in the following theorem is chosen in this way.

\begin{corollary}\label{PropToric}
   Let $I$ be an ideal of $\A$, the coordinate ring of a pointed normal affine toric variety over a field $\KK$ of characteristic zero.
   Then there exists an embedding of $\A$ into a polynomial ring $\R$ over $\KK$ such that   
$
\cJ_\A (I^\lambda)=\cJ_\R((I\R)^\lambda)\cap \A
$ 
for all real numbers $\lambda\geq 0$.
\end{corollary}

\begin{proof}
With $\R$ chosen as in the discussion before the statement, $\A_p \subseteq \R_p$ is Cartier extensible for $p\gg 0$ by \Cref{LDE->CE}, so  we can apply \Cref{theorem-multiplier}.
\end{proof}

The following result is already established for monomial ideals $I$  \cite{HsiaoMatusevich}.

\begin{corollary}\label{ThmToric}
	Let $I\subseteq \A$ be an ideal in an affine normal toric ring over an algebraically closed field $\KK$
	of characteristic zero. 
	Then the log canonical threshold of 
	$ I$ in $\A$ coincides with the smallest root $\alpha_I$ of the Bernstein--Sato polynomial 
	$b_I(-s)$, and any jumping numbers of $I$ in $[\alpha_I,\alpha_I +1)$ are roots of $b_I(-s)$. 
	\qed
\end{corollary}

We obtain similar results  for rings of invariants under the action of a finite group.

\begin{corollary}\label{ThmGroup}
 Let $\KK$ be a field of characteristic zero, and  $G$ a finite group acting linearly on a polynomial ring $\R$ over $\KK$. Suppose that $G$ contains no element that fixes a hyperplane in the space of one-forms $[\R]_1$. Let $\A=\R^G$ denote the ring of invariants. 
 Then the log canonical threshold of 
 $ I$ in $\A$ coincides with the smallest root $\alpha_I$ of 
 $b_I(-s)$, and every jumping number of $I$ in $[\alpha_I,\alpha_I +1)$ is a root of $b_I(-s)$.
 \qed
\end{corollary}

In the polynomial ring setting, we can understand Hodge ideals as a generalization of multiplier ideals. Indeed,
$$ I_0 ^R(f^\lambda) =  \cJ_\R(f^{\lambda-\varepsilon})$$
for $\varepsilon >0$ small enough  \cite[Proposition~9.1]{MPBirational}. In the case of ideals $I\subseteq R$, the same statement holds  when $\lambda \in (0,1]\cap \QQ$  \cite{MPIdeals}.
As a consequence of \Cref{theorem-multiplier}, we have the same property for extensible direct summands of polynomial rings.

\begin{proposition}\label{prop-hodge-multiplier}
	In the context of \Cref{setup: multiplier}, for every  $\lambda \in \QQ_{\geq 0}$ and $\varepsilon >0$ small enough, we have
	\[
	I_0^A(f^\lambda)= \cJ_\A (f^{\lambda-\varepsilon}).
	\]
\end{proposition}

\begin{proof}
	Using \Cref{CorHodge} and \Cref{theorem-multiplier} we have
	\[
	I_0^A(f^\lambda)= I^R_0(f^\lambda)\, \cap \, A  =  \cJ_\R (f^{\lambda-\varepsilon})\, \cap \, A 
	=            \cJ_\A (f^{\lambda-\varepsilon}).\qedhere
	\]
\end{proof}

\section*{Acknowledgments}
This project was initiated and developed during an American Institute of Mathematics SQuaRE (Structured Quartet Research Ensemble) titled \emph{Roots of Bernstein--Sato polynomials}. 
We are grateful for AIM's hospitality, and for their support of this project. 
We  thank Devlin Mallory for helping us compute \Cref{ExDetl3x3}, and Eamon Quinlan for a careful reading of the manuscript. We also thank Nero Budur, Manuel Gonz\'alez, Craig Huneke, Mircea Musta\c{t}\u{a}, Luis Narv\'aez-Macarro, Claude Sabbah, and Karl Schwede for helpful conversations and correspondences.
In particular, we thank  Manuel Gonz\'alez for bringing our attention to Hodge ideals, and for discussions regarding Subsection \ref{Hodge}.
\bibliographystyle{myalpha}
\bibliography{refs}

\end{document}